\documentclass[final,reqno]{siamltex}
\usepackage{latexsym,amsmath,amssymb,amsfonts,mathrsfs}
\usepackage{epsf,graphicx,epsfig,color,cite,cases}
\usepackage{subfigure,graphics,multirow,marginnote,enumerate,bm}
\newcommand{\Div}{{\rm Div\,}}

\newcommand{\Curl}{{\rm Curl\,}}
\newcommand{\g}{\gamma}
\newcommand{\G}{\Gamma}

\newcommand{\vep}{\varepsilon}
\newcommand{\Om}{\Omega}
\newcommand{\om}{\omega}

\newcommand{\ka}{\kappa}

\newcommand{\la}{\lambda}

\newcommand{\wi}{\widetilde}

\newcommand{\Int}{\int\limits}

\newcommand{\pa}{\partial}

\newcommand{\ov}{\overline}
\newcommand{\I}{{\rm Im}}
\newcommand{\Rt}{{\rm Re}}
\newcommand{\curl}{{\rm curl\,}}
\newcommand{\dive}{{\rm div\,}}

\newcommand{\ti}{\times}

\newcommand{\na}{\nabla}
\newcommand{\mat}{\mathbb}
\newcommand{\se}{\setminus}

\newcommand{\R}{{\mat R}}

\newcommand{\N}{{\mat N}}
\newcommand{\C}{{\mat C}}
\newcommand{\Sp}{{\mat S}}
\newcommand{\ds}{\displaystyle}
\newcommand{\no}{\nonumber}
\newcommand{\be}{\begin{eqnarray}}
\newcommand{\ben}{\begin{eqnarray*}}
\newcommand{\en}{\end{eqnarray}}
\newcommand{\enn}{\end{eqnarray*}}

\begin{document}
\renewcommand{\theequation}{\arabic{section}.\arabic{equation}}

\title{\bf
On recovery of a bounded elastic body by electromagnetic far-field measurements}
\author{Tielei Zhu\thanks{School of Mathematics and Statistics, Xi'an Jiaotong University,
Xi'an 710049, Shaanxi, China ({\tt zhutielei@stu.xjtu.edu.cn})}
\and
Jiaqing Yang\thanks{School of Mathematics and Statistics, Xi'an Jiaotong University,
Xi'an 710049, Shaanxi, China ({\tt jiaq.yang@mail.xjtu.edu.cn})}
\and
Bo Zhang\thanks{NCMIS, LSEC and Academy of Mathematics and Systems Sciences, Chinese Academy of Sciences,
Beijing 100190, China and School of Mathematical Sciences, University of Chinese Academy of Sciences,
Beijing 100049, China ({\tt b.zhang@amt.ac.cn})}
}
\date{}
\maketitle


\begin{abstract}
This paper is concerned with the problem of scattering of a time-harmonic electromagnetic field by
a three-dimensional elastic body. General transmission conditions are considered to model the
interaction between the electromagnetic field and the elastic body on the interface by assuming Voigt's model.
The existence of a unique solution of the interaction problem is proved in an appropriate Sobolev space
by employing a variational method together with the classical Fredholm alternative.
The inverse problem is then considered, which is to recover the elastic body by the scattered wave-field.
It is shown that the shape and location of the elastic body can be uniquely determined by the fixed energy
magnetic (or electric) far-field measurements corresponding to incident plane waves with all polarizations.
\end{abstract}



\pagestyle{myheadings}
\thispagestyle{plain}
\markboth{T. Zhu, J. Yang and B. Zhang}{Recovery of elastic body by electromagnetic far-field measurements}

\section{Introduction}\label{sec1}
\setcounter{equation}{0}

The interaction of different physical fields has received considerable attention due to the rapidly
increasing use of composite materials. Therefore, it is significant to develop the related mathematical
model and analysis by physical process. The physical kinematic and dynamic relations are described by
the corresponding partial differential equations (PDEs) with certain boundary-transmission conditions.
Generally, it is difficult to find an appropriate interaction condition connected with different physical
fields on the interface.

For time-harmonic acoustic wave scattering by a solid body, many work has been done on the mathematical 
analysis of the interaction problem (see, e.g., \cite{CP95,GRG00}). Recently, the corresponding inverse 
problems have also been studied mathematically and numerically of detecting an elastic body via the 
measurement of the acoustic scattered wave field. We refer the reader to \cite{DSZ00,PV09,FJB17,AR12} 
for detailed discussions.
In particular, it is shown in \cite{DSZ00} that a uniqueness result was first proved in recovering an
elastic body by the acoustic far-field measurements. The proof was then simplified by Monk and Selgas \cite{PV09}
by using the technique of H\"{a}hner in \cite{P00} for the case of a penetrable, anisotropic obstacle.
However, the analysis in \cite{PV09} relies on the $H^2$-regularity estimate of solutions of the scattering
problem, and thus the proposed method remains complicated. Very recently, a much simpler proof
was introduced by Qu {\em et al.} \cite{FJB17}, which is motivated by the previous work of
the last two authors \cite{JBH18} for inverse acoustic and electromagnetic scattering by a penetrable obstacle,
and can be extended to deal with other more general cases.

In this paper we consider the problem of scattering of a time-harmonic electromagnetic field by a 
three-dimensional elastic body. Assume by Voigt's model that the interaction is allowed only through
the boundary of the body, which means that the model problem can be described by the Maxwell and Navier 
equations coupled with a suitable transmission condition on the interface between
the elastic and electromagnetic medium. It was shown \cite{FG02} that an interaction model was first
introduced by Cakoni and Hsiao with possible interface conditions for the coupled electromagnetic and
elastic fields, where the uniqueness result and equivalent integral equations and
non-local variational formulations have been established for the model.
Applying the framework in \cite{FG02}, Gatica {\em et al.} \cite{GGS10} proved the existence of a unique 
solution of the interaction problem by using a variational method. The result was later extended by
Bernardo {\em et al.} \cite{AAS10} to a different function space for the elastic field,
based on a similar idea to \cite{GGS10}.

Different from \cite{AAS10} and \cite{GGS10}, we study in this paper the interaction problem with
general interface transmission conditions which could model more physical situations in applications.
An equivalent non-symmetric variational formulation is then obtained by using Green's formulas so that 
the existence of a unique solution to the problem can follow from the classical Fredholm alternative 
with a suitable Helmholtz-type decomposition of the electromagnetic field. 
Compared with the forward problem, the inverse problem of determining the elastic body is more challenging 
due to the complication of the interaction model.
To the best of our knowledge, no uniqueness result is available for this problem in the literature.
Inspired by our previous work \cite{JBH18} where a novel technique was introduced for showing uniqueness 
in determining an acoustic or electromagnetic penetrable obstacle, we aim to develop a novel and simple 
technique to prove the unique recovery of the elastic body by the electromagnetic far-field measurement 
at a fixed frequency. The proposed method is mainly based on constructing a well-posed system of PDEs for 
the coupled Maxwell and Navier equations in a small domain near the interface in conjunction with a 
uniform a priori estimate in the $H(\curl,\cdot)\times {\bm H}^1(\cdot)$ norm of solutions to the 
interaction problem when the incident electromagnetic fields are induced by a family of electric dipoles 
with a weak singularity.

The remaining part of the paper is organized as follows. In Section \ref{sec2}, we formulate the interaction 
scattering problem by collecting some useful functions spaces, trace operators and related properties. 
In Section \ref{sec3}, we show the existence of a unique solution to the interaction problem by 
the variational method with aid of a suitable Helmholtz-type decomposition. In Section \ref{sec4}, 
a global uniqueness theorem is proved for the associated inverse problem of determining the elastic body 
from the magnetic or electric far-field measurements at a fixed frequency.

\section{The model problem}\label{sec2}

In this section, we first introduce some basic notations and function spaces used throughout this paper and  
then present the mathematical formulation of the model problem.

\subsection{Preliminaries} \label{sec2.1}

For a complex number $z\in\mat{C}$, its conjugate and modulus are denoted by $\ov{z}$
and $|z|$, respectively. For $x,y\in\mat{C}^3$ define $x\cdot y=\sum_{j=1}^3x_jy_j$.
Let $D\in\R^3$ be a bounded domain with a $C^2$-boundary $\pa D$ and let $\nu$ be the unit outward 
normal to $\pa D$. For $s\in\mat{R}$ denote by $H^s(D)$ and $H^s(\pa D)$ the standard scalar Sobolev spaces 
defined on $D$ and $\pa D$, respectively, with $L^2(\cdot):=H^0(\cdot)$. We also need the following vector 
function spaces defined on $D$ and $\pa D$:
\ben
{\bm H}^s(D)&:=&[H^s(D)]^3,\quad s\in\mat{R}, \\
{\bm H}^s(\pa D)&:=&[H^s(\pa D)]^3,\quad s\in\mat{R}, \\
{\bm H}^s_t(\pa D)&:=&\{\mu\in{\bm H}^s(\pa D): \mu\cdot\nu=0 \},\quad s\in\mat{R},\\
H(\curl,D)&:=&\{ {\bm B}\in {\bm L}^2(D): \curl{\bm B}\in {\bm L}^2(D) \}.
\enn
For each $s\in\R$, $\langle H^{-s}, H^s\rangle$ and $\langle{\bm H}^{-s}, {\bm H}^s\rangle$ denote 
the duality product under the extension of the ${\bm L}^2$-bilinear form
\ben
\langle{\bm u},{\bm v}\rangle:=\Int_D{\bm u}\cdot{\bm v}{\,\rm d}x,\;\;\;{\bm u},\; {\bm v}\in{\bm L}^2(D).
\enn

By \cite{P03} the tangential trace spaces of $H(\curl,D)$ can be characterized as
\ben
H^{-{1}/{2}}_{\Div}(\pa D)&:=&\{\mu\in{\bm H}^{-{1}/{2}}_t(\pa D):\Div_{\pa D}\mu\in H^{-{1}/{2}}(\pa D)\},\\
H^{-{1}/{2}}_{\Curl}(\pa D)&:=&\{\mu\in{\bm H}^{-{1}/{2}}_t(\pa D):\Curl_{\pa D}\mu\in H^{-{1}/{2}}(\pa D)\},
\enn
where $\Div_{\pa D}$ and $\Curl_{\pa D}$ denote the surface divergence and surface curl with respect to 
the boundary $\pa D$, respectively. For convenience, we also use $\na_{\pa D}\cdot$ to denote $\Div_{\pa D}$, 
which is the surface gradient defined by $\na_{\pa D} f:=(\nu\times\na f)\times\nu$ for a smooth function $f$. 
For a smooth vector function ${\bm u}\in [C(\ov{D})]^3$, we introduce the tangential trace mapping $\g_t$ 
and the tangential projection operator $\g_T$ by
\ben
&&\;\g_t{\bm u}: = \nu\ti{\bm u}\qquad\quad {\rm on\;}\pa D,\\
&&\g_T{\bm u}: = \nu\ti({\bm u}\ti\nu)\qquad {\rm on\;}\pa D,
\enn
which can be extended as bounded and surjective operators from $H(\curl, D)$ into $H^{-{1}/{2}}_{\Div}(\pa D)$ 
and $H^{-{1}/{2}}_{\Curl}(\pa D)$, respectively.

Further, the duality product between $H^{-{1}/{2}}_{\Div}(\pa D)$ and $H^{-{1}/{2}}_{\Curl}(\pa D)$ is defined as
\ben\no
\left\langle\varphi,\psi\right\rangle_{H^{-1/2}_{\Div}\times{H}_{\Curl}^{-{1}/{2}}}
:=\Int_{D}\left[\curl{\bm v}\cdot{\bm w}-\curl{\bm w}\cdot{\bm v}\right]{\,\rm d}x
\enn
for ${\bm v},{\bm w}\in H(\curl,D)$ satisfying $\g_t{\bm v}=\varphi$ and $\g_T{\bm w}=\psi$. 
In addition, introduce the tangential trace spaces of ${\bm H}^1(D)$:
\ben
{\bm H}^{{1}/{2}}_{\perp}(\pa D):=\g_t({\bm H}^1(D))\quad{\rm and}\quad {\bm H}^{{1}/{2}}_{||}(\pa D)
:=\g_T({\bm H}^1(D)),
\enn
and denote by ${\bm H}^{-1/2}_{\perp}(\pa D)$ and ${\bm H}^{-1/2}_{||}(\pa D)$ the dual space of 
${\bm H}^{{1}/{2}}_{\perp}(\pa D)$ and ${\bm H}^{{1}/{2}}_{||}(\pa D)$, respectively. 
Clearly, $H^{-{1}/{2}}_{\Div}(\pa D)$ is a closed subspace of ${\bm H}^{-{1}/{2}}_{||}(\pa D)$, 
which can be understood in the sense that 
\be\no
\left\langle\varphi,\psi\right\rangle_{{\bm H}^{-1/2}_{\Div}\times{\bm H}^{1/2}_{||}}
:=\Int_{D}\left[\curl{\bm w}\cdot{\bm v}-\curl{\bm v}\cdot{\bm w}\right]{\,\rm d}x
\en
for ${\bm w}\in H(\curl,D)$ and ${\bm v}\in {\bm H}^1(D)$ satisfying 
$\g_t{\bm w}=\varphi\in H^{-1/2}_{\Div}(\pa D)$ and $\g_T{\bm v}=\psi\in{\bm H}^{1/2}_{||}(\pa D)$.

\subsection{The mathematical formulation}\label{sec2.2}

In this subsection, we formulate the mathematical formulation of the problem of scattering of a time-harmonic electromagnetic wave by an elastic body in $\R^3$. As seen in Figure \ref{fig:label}, the elastic body is 
described by a bounded domain $D$ with a smooth boundary of $C^2$-class, which is assumed to be inhomogeneous 
and anisotropic with the stiffness tensor $\mathcal{C}:=(C_{ijkl}(x))_{i,j,k,l=1}^3$, 
where $C_{ijkl}\in L_\infty(D)$ ($i,j,k,l=1,2,3$), and the density $\rho(x)\in L_\infty(D)$. 
The background medium outside of $D$ is assumed to be homogeneous and
isotropic with constant electric permittivity $\vep_0\in\R_+$ and magnetic permeability $\mu_0\in\R_+$.

\setcounter{equation}{0}
\begin{figure}[htp]
	\centering
	\includegraphics[scale=0.45]{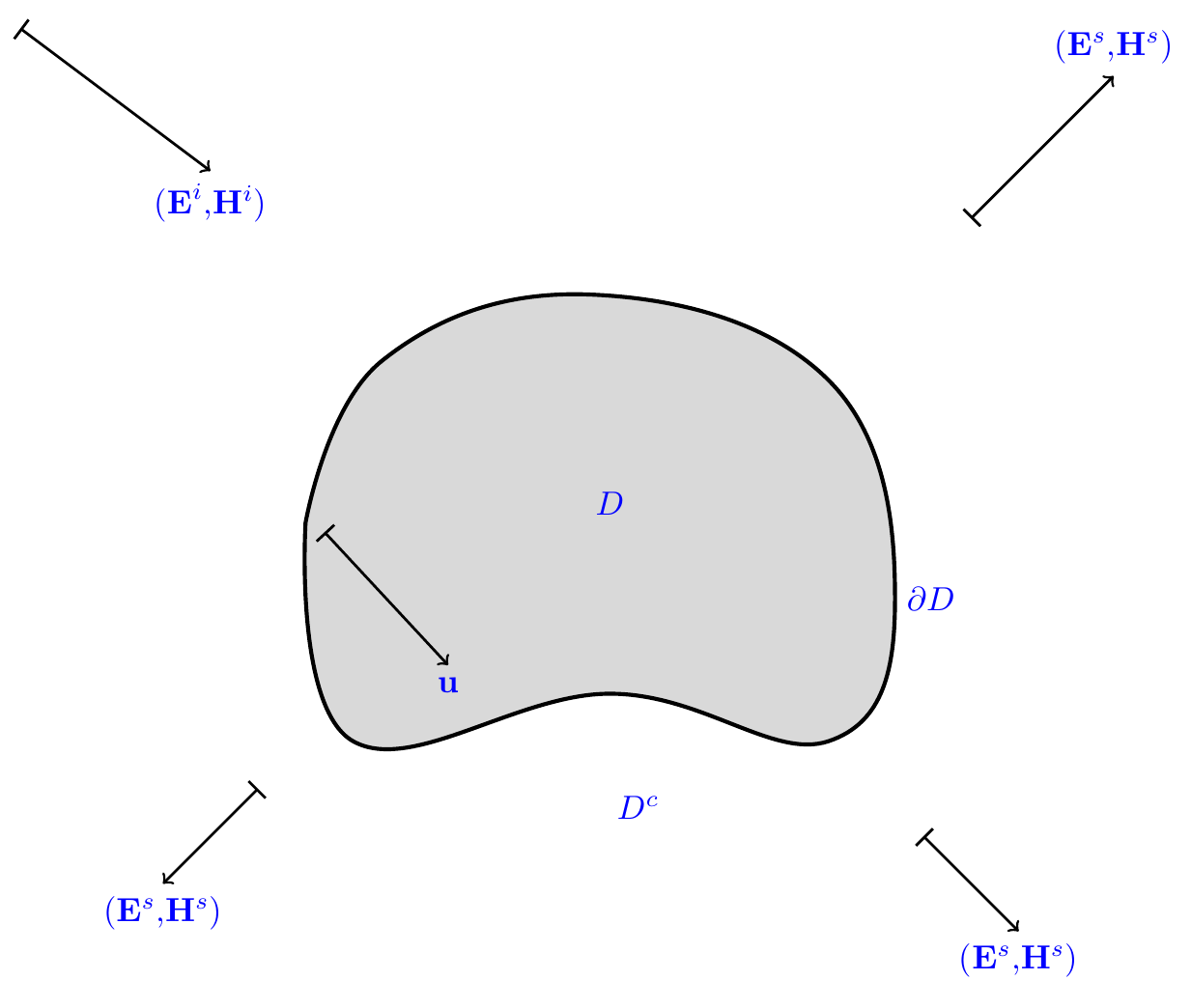}
	\caption{Interaction between electromagnetic wave and a bounded elastic body}
	\vskip -0.8truecm
	\hspace{0.2in}
	\label{fig:label}
\end{figure}

Consider a pair of electromagnetic waves of the form
\be\label{2.1a}
{\bm E}^i(x,d,p)=-\frac{1}{\rm i\ka}\curl^2(pe^{{\rm i}\ka x\cdot d}), 
\quad {\bm H}^i(x,d,p)=\curl(pe^{{\rm i}\ka x\cdot d}),\quad\;x\in\R^3
\en
which is incident on $D$ from the unbounded exterior domain $D^c:=\R^3\se\ov{D}$, 
where $\om\in\R_+$ is the wave frequency,
$d\in\mat{S}^2$ is the direction of wave propagation, $p\in\mat{R}^3$ is the polarization vector 
and $\ka=\om\sqrt{\vep_0\mu_0}$ is the wave number. Then the elastic deformation occurs due to the physical 
property of the elastic body. Following Voigt's model, we can assume that the electromagnetic wave does not 
penetrate the elastic body so that the interaction occurs only on the interface. Under the above physical 
assumption, the elastic field ${\bm u}$ satisfies the Navier equation
\be\label{2.2a}
\na\cdot(\mathcal{C}:\na{\bm u})+\rho\om^2{\bm u}=0\qquad {\rm in}\;\;\; D,
\en
where $\mathcal{C}:\na{\bm u}$ is defined as
\ben
\mathcal{C}:\na{\bm u}:=(\mathcal{C}:\na{\bm u})_{ij}=\sum\limits_{k,l=1}^3C_{ijkl}\frac{\pa u_k}{\pa x_l}
\enn
and $\na\cdot( \mathcal{C}:\na{\bm u})$ is defined as
\ben
\na\cdot(\mathcal{C}:\na{\bm u})=\sum\limits_{j,k,l=1}^3\frac{\pa}{\pa x_j}(C_{ijkl}\frac{\pa u_k}{\pa x_l}).
\enn
Moreover, the stiffness tensor $\mathcal{C}$ in (\ref{2.1a}) satisfies the symmetry condition
\ben
C_{ijkl}=C_{klij}=C_{jikl}=C_{ijlk}
\enn
and the Legendre elliptic condition
\ben
\sum\limits_{i,j,k,l=1}^3C_{ijkl}(x)a_{ij}a_{kl}\geq c_0\sum\limits_{i,j=1}^3|a_{ij}|^2,\quad
a_{ij}=a_{ji}
\enn
for some positive constant $c_0>0$.

In the exterior domain $D^c$, the electromagnetic field $({\bm E},{\bm H})$, which is the sum of 
the incident field $({\bm E}^i,{\bm H}^i)$ and the scattered field $({\bm E}^s,{\bm H}^s)$, 
satisfies the Maxwell equations
\be\label{2.3a}
\curl{\bm E}-{\rm i}\ka{\bm H}=0,
\qquad \curl{\bm H}+{\rm i}\ka{\bm E}=0 \qquad {\rm in}\;\;\;D^c,
\en
with the scattered field $({\bm E}^s,{\bm H}^s)$ satisfying the Silver-M\"{u}ller radiation condition
\be\label{2.4a}
\lim_{|x|\to\infty}|x|({\bm E}^s\times\hat{x}+{\bm H}^s)=0,
\en
where $\hat{x}=x/|x|\in\Sp^2$.

On the interface, the electromagnetic and elastic fields are assumed to be coupled by the general 
transmission conditions (cf. \cite{FG02}):
\be\label{2.5a}
T{\bm u}-b_1\nu\times{\bm H}^s&=&b_1\nu\times{\bm H}^i\qquad {\rm on}\;\;\;\pa D,\\ \label{2.5b}
\nu \ti{\bm u}-b_2\nu \ti{\bm E}^s&=&b_2\nu\times{\bm E}^i\qquad {\rm on}\;\;\;\pa D,
\en
where $b_1$, $b_2\in\C$, $b_1b_2\not=0$, $\nu$ is the unit outward normal to $\pa D$ and 
$T$ is defined as
\ben
(T{\bm u})_i=\sum\limits_{j,k,l=1}^3\nu_jC_{ijkl}\frac{\pa u_k}{\pa x_l},\;\;\;\;i=1,2,3.
\enn
We refer the reader to \cite{FG02,GGS10} for detailed discussions on different choices of $b_1$ and $b_2$.

By the radiation condition (\ref{2.4a}), it is well-known that the scattered field has the asymptotic behavior
\ben\no
{\bm H}^s(x,d,p)&=&\frac{e^{{\rm i}\ka|x|}}{|x|}{\bm H}^s_{\infty}(\hat{x},d,p)+O\left(\frac{1}{|x|^2}\right),
\quad{\rm as}\;\;\;|x|\to\infty\\
{\bm E}^s(x,d,p)&=&\frac{e^{{\rm i}\ka|x|}}{|x|}{\bm E}^s_{\infty}(\hat{x},d,p)+O\left(\frac{1}{|x|^2}\right),
\quad {\rm as}\;\;\; |x|\to\infty,
\enn
where ${\bm E}^s_{\infty}(\hat{x},d,p)$ and ${\bm H}^s_{\infty}(\hat{x},d,p)$ denote the electric and magnetic 
far-field patterns, respectively, which are analytic in $\hat{x}\in\Sp^2$ and $d\in\Sp^2$, respectively, 
and satisfy the relations (cf. \cite{DR13}):
\be\label{2.6a}
{\bm H}^s_{\infty}(\hat{x},d,p)=\hat{x}\ti{\bm E}^s_{\infty}(\hat{x},d,p),\quad
\hat{x}\cdot{\bm E}^s_{\infty}(\hat{x},d,p)=0,\quad
\hat{x}\cdot{\bm H}^s_{\infty}(\hat{x},d,p)=0.\;\;
\en

\section{The well-posedness of the interaction problem}\label{sec3}
\setcounter{equation}{0}

In this section, we prove the well-posedness of the interaction problem (\ref{2.1a})-(\ref{2.5a}), 
employing a variation method. 
Under the transmission conditions (\ref{2.5a}) and (\ref{2.5b}), 
the existence of a unique solution can be obtained by showing the variational formulation to be
of Fredholm with index $0$ in an appropriate Sobolev space.

\subsection{Uniqueness of solutions}

As known for the fluid-solid interaction problem, non-uniqueness may exist for certain frequencies 
which are called Jones frequencies. Similarly, there may exist pathological frequencies in the interaction 
between an electromagnetic wave and an elastic body so that a nontrivial solution exists for the 
homogeneous problem corresponding to the problem (\ref{2.1a})-(\ref{2.5a}). 
Thus, introduce the homogeneous problem
\be\label{3.1a}
\nabla\cdot(\mathcal{C}:\nabla{\bm u})+\rho\om^2{\bm u}&=&0\qquad{\rm in}\;\;\;D,\\ \label{3.1b}
\nu \ti{\bm u}&=&0 \qquad{\rm on}\;\;\;\pa D,\\ \label{3.1c}
T{\bm u}&=&0 \qquad{\rm on}\;\;\; \pa D
\en
and let the set $\mathcal{P}(\om)$ be consisting of the frequency $\om\in\R$ such that (\ref{3.1a}) 
has a nontrivial solution. 
Then we have the following result on uniqueness of solutions to the problem (\ref{2.1a})-(\ref{2.5a}).

\begin{theorem}\label{thm3.1}
Assume that $\rho,\ka,\om\in\R$ and $\om\notin\mathcal{P}(\om)$. If $\Rt(b_1\ov{b}_2)=0$, then 
the scattering problem $(\ref{2.1a})-(\ref{2.5a})$ has at most one solution.
\end{theorem}

\begin{proof}
Let ${\bm E}^i={\bm H}^i=0$. Then it is enough to prove that ${\bm E}^s={\bm H}^s=0$.
Using Green's formula and the transmission conditions (\ref{2.5a}) and (\ref{2.5b}), we have
\ben
\Int_{\pa D}\nu\ti\ov{{\bm E}^s}\cdot{\bm H}^s{\,\rm d}s
=-\frac{1}{b_1\ov{b}_2}\Int_{\pa D}T{\bm u}\cdot\ov{\bm u}{\,\rm d}s 
=-\frac{1}{b_1\ov{b}_2}\Int_{D}[\mathcal{E}({\bm u},\ov{{\bm u}})-\rho\om^2|{\bm u}|^2]{\,\rm d}x,
\enn
where
\ben
\mathcal{E}({\bm u},{\bm v}):=\sum\limits_{i,j,k,l=1}^3C_{ijkl}\frac{\pa u_k}{\pa x_l}\frac{\pa v_i}{\pa x_j}
\enn
for ${\bm u},{\bm v}\in {\bm H}^1(D)$. Thus
\ben
\Rt \Int_{\pa D}\nu\ti\ov{{\bm E}^s}\cdot{\bm H}^s{\,\rm d}s
=-\Rt\frac{1}{b_1\ov{b}_2}\Int_{D}[\mathcal{E}({\bm u},\ov{{\bm u}})-\rho\om^2|{\bm u}|^2]{\,\rm d}x.
\enn
By this, Rellich's lemma (see \cite{DR13}) and the fact that $\Rt(b_1\ov{b}_2)=0$, we conclude
that ${\bm E}^s={\bm H}^s=0$. Then the elastic field ${\bm u}$ satisfies (\ref{3.1a}), yielding ${\bm u}=0$ 
since $\om\notin\mathcal{P}(\om)$.
\end{proof}

\subsection{Existence of solutions}

We now prove the existence of solutions of the scattering problem (\ref{2.1a})-(\ref{2.5a}),  
employing a variational method. To this end, we eliminate the electric field ${\bm E}$ and consider 
the boundary value problem for $({\bm H},{\bm u})$:
\be\label{3.2a}
\left\{\begin{array}{ll}
\ds\na\cdot(\mathcal{C}:\na{\bm u})+\rho\om^2{\bm u}=0 & \quad {\rm in}\;\;\;D,\\[1mm]
\ds\curl\curl{\bm H}-\ka^2{\bm H}=0  & \quad{\rm in}\;\;\; D^c,\\[1mm]
\ds T{\bm u}- b_1\nu\times{\bm H}={\bm f_1} & \quad {\rm on}\;\;\; \pa D,\\[1mm]
\ds \nu\times{\bm u}+\frac{b_2}{{\rm i}\ka}\nu\times\curl{\bm H}={\bm f_2} & \quad{\rm on}\;\;\;\pa D,\\[1mm]
\ds\lim_{r\to\infty}r(\hat{x}\times\curl{\bm H}+{\rm i}\ka{\bm H})=0  &\quad r=|x|,
\end{array}
\right.
\en
where ${\bm f}_1\in{\bm H}^{-{1}/{2}}(\pa D)$ and ${\bm f_2}\in H^{-{1}/{2}}_{\Div}(\pa D)$.
Note that the scattering problem (\ref{2.1a})-(\ref{2.5a}) can be viewed as a special case of 
the problem (\ref{3.2a}) with ${\bm f}_1:=b_1\nu\times{\bm H}^i$ and 
${\bm f}_2:=(b_2/{\rm i}\ka)\nu\times\curl{\bm H}^i$. 

We now reduce the problem (\ref{3.2a}) into one in the bounded domain $B_R:=\{x\in\R^3:|x|\leq R\}$ 
with $R$ large enough. To this end, we introduce the Calder\'{o}n mapping 
$G_e: H^{-{1}/{2}}_{\Div}(S_R)\to H^{-{1}/{2}}_{\Div}(S_R)$ defined by
\be\label{3.2aa}
G_e(\la):=\frac{1}{{\rm i}\ka}\hat{x}\ti\curl\wi{{\bm w}}\qquad{\rm on}\;\;\;S_R:=\pa B_R
\en
for $\la\in H^{-{1}/{2}}_{\Div}(S_R)$, where $\wi{{\bm w}}$ satisfies the problem 
\ben
\left\{\begin{array}{ll}
\ds\curl\curl\wi{{\bm w}}-\ka^2\wi{{\bm w}}=0& \quad {\rm in\;}\;\;\R^3\se\ov{B_R},\\[1mm]
\ds\hat{x}\ti\wi{{\bm w}}=\la &\quad {\rm on\;}\;\; S_R,\\[1mm]
\ds\lim_{r\to\infty} r(\hat{x}\times\curl\wi{{\bm w}}+{\rm i}\ka\wi{{\bm w}})=0& \quad r=|x|.
\end{array}
\right.
\enn
The Calder\'{o}n mapping $G_e$ has the following properties which were proved in \cite{P03}.

\begin{lemma}\label{lem3.3} 
Let $\wi{G}_e$ be defined as $G_e$ in $(\ref{3.2aa})$ with $\ka={\rm i}$. Then
\begin{enumerate}
\item[{\rm(a)}] $G_e+{\rm i}\ka\wi{G}_e$ is compact from $H^{-1/2}_{\Div}(S_R)$ to $H^{-1/2}_{\Div}(S_R)$;
\item[{\rm(b)}] $\langle\wi{G}_e\la,\ov{\la}\times\hat{x}\rangle_{H^{-1/2}_{\Div}\times H^{-1/2}_{\Curl}}<0$ 
for any $\la\in H^{-1/2}_{\Div}(S_R)$ with $\la\neq 0$;
\item[{\rm(c)}] $G_e$ can be split as $G_e\la:=G_e^1\la+G_e^2\la$ for $\la\in H^{-1/2}_{Div}(S_R)$ 
such that
\begin{enumerate}
\item[{\rm(c.1)}] The map ${\bm H}\mapsto G_e^1(\hat{x}\ti {\bm H})$ is compact from $X_0$ into 
$H^{-1/2}_{Div}(S_R)$, where $X_0$ is defined in (\ref{3.6cc}) below;
\item[{\rm(c.2)}] ${\rm i}\ka\langle G_e^2(\hat{x}\times{\bm H}),\g_T\ov{\bm H}\rangle_{H^{-1/2}_{\Div}
\times H^{-1/2}_{\Curl}}\ge 0$ for ${\bm H}\in H(\curl,B_R)$.
\end{enumerate}
\end{enumerate}
\end{lemma}

With the aid of the Calder\'{o}n map $G_e$, the problem (\ref{3.2a}) can be equivalently reduced to 
the boundary value problem
\be\label{3.3a}
\left\{
\begin{array}{ll}
\ds\na\cdot(\mathcal{C}:\na{\bm u})+\rho\om^2 {\bm u}=0 &\quad {\rm in\;}\;\; D,\\[1mm]
\ds\curl\curl {\bm H} -\ka^2{\bm H}=0  &\quad {\rm in\;}\;\; B_R\se\ov{D},\\[1mm]
\ds T{\bm u}- b_1\nu \ti{\bm H}={\bm f_1}  & \quad {\rm on\;}\;\; \pa D,\\[1mm]
\ds\nu\times{\bm u}+\frac{b_2}{{\rm i}\ka}\nu\times\curl{\bm H}={\bm f_2}  &\quad {\rm on\;}\;\; \pa D,\\[1mm]
\ds\hat{x}\times\curl{\bm H}={\rm i}\ka G_e(\hat{x}\times{\bm H}) &\quad {\rm on\;}\;\; S_R,
\end{array}
\right.
\en
where ${\bm f}_1\in {\bm H}^{-1/2}(\pa D)$ and ${\bm f_2}\in H^{-1/2}_{\Div}(\pa D)$.

Multiplying the first equation of (\ref{3.3a}) with ${\bm v}\in{\bm H}^1(D)$ and using integration by parts 
together with the third equation of (\ref{3.3a}) yield 
\ben
\Int_D(\mathcal{E}({\bm u},\ov{{\bm v}})-\rho\om^2{\bm u}\cdot\ov{\bm v} ){\,\rm d}x
- b_1\langle \g_t{\bm H} ,\g_T\ov{\bm v}\rangle_{H^{-\frac{1}{2}}_{\Div}\ti H^{-\frac{1}{2}}_{\Curl}}
=\langle {\bm f}_1,\ov{\bm v}\rangle_{{\bm H}^{-\frac{1}{2}}\ti {\bm H}^{\frac{1}{2}}}
\enn
or equivalently 
\ben
\frac{-{\rm i}\ka}{b_1\ov{b}_2}\Int_D(\mathcal{E}({\bm u},\ov{{\bm v}})
-\rho\om^2{\bm u}\cdot\ov{\bm v} ){\,\rm d}x
+\frac{{\rm i}\ka}{\ov{b}_2}\langle\g_t{\bm H},\g_T\ov{\bm v}
\rangle_{H^{-\frac{1}{2}}_{\Div}\times H^{-\frac{1}{2}}_{\Curl}}
=\frac{-{\rm i}\ka}{b_1\ov{b}_2}\langle{\bm f}_1,\ov{\bm v}
\rangle_{{\bm H}^{-\frac{1}{2}}\times{\bm H}^{\frac{1}{2}}}
\enn
for all ${\bm v}\in{\bm H}^1(D)$. 
Multiplying the second equation of (\ref{3.3a}) by $-{\bm w}\in H(\curl,B_R\se\ov{D})$ 
and utilizing the fourth and fifth equations of (\ref{3.3a}) give
\ben
&&-\Int_{B_R\se\ov{D}}(\curl{\bm H}\cdot\curl\ov{{\bm w}}-\ka^2{\bm H}\cdot\ov{{\bm w}}){\,\rm d}x
-{\rm i}\ka\langle G_e(\hat{x}\times{\bm H}),\g_T\ov{\bm w}
\rangle_{H^{-\frac{1}{2}}_{\Div}\times H^{-\frac{1}{2}}_{\Curl}}\\
&&\qquad+\frac{{\rm i}\ka}{b_2}\langle\g_t\ov{\bm w},\g_T{\bm u}
\rangle_{H^{-\frac{1}{2}}_{\Div}\times H^{-\frac{1}{2}}_{\Curl}}
=-\frac{{\rm i}\ka}{b_2}\langle{\bm f}_2,\g_T\ov{\bm w}
\rangle_{H^{-\frac{1}{2}}_{\Div}\times H^{-\frac{1}{2}}_{\Curl}}.
\enn
Adding the above two equations together and letting $X=H(\curl,B_R\se\ov{D})$ and $Q={\bm H}^1(D)$,
we obtain the variational formulation of (\ref{3.3a}): find $({\bm u},{\bm H})\in Q\times X$ such that
\be\label{3.4a}
\mathcal{A}(({\bm u},{\bm H}),({\bm v},{\bm w}))=\mathcal{F}({\bm v},{\bm w})\;\;\;\;\;
\forall ({\bm v},{\bm w})\in Q\times X,
\en
where
\ben
\mathcal{A}(({\bm u},{\bm H}),({\bm v},{\bm w}))
&:=&\frac{-{\rm i}\ka}{b_1\ov{b}_2}\Int_D(\mathcal{E}({\bm u},\ov{{\bm v}})
-\rho\om^2{\bm u}\cdot\ov{\bm v}){\,\rm d}x
+\frac{{\rm i}\ka}{\ov{b}_2}\langle\g_t{\bm H},\g_T\ov{\bm v}
\rangle_{H^{-\frac{1}{2}}_{\Div}\times H^{-\frac{1}{2}}_{\Curl}}\\
&&-\Int_{B_R\setminus\ov{D}}(\curl{\bm H}\cdot\curl\ov{{\bm w}}-\ka^2{\bm H}\cdot\ov{{\bm w}}){\,\rm d}x\\
&&-{\rm i}\ka\langle G_e(\hat{x}\ti{\bm H}),\g_T\ov{\bm w}
\rangle_{H^{-\frac{1}{2}}_{\Div}\times H^{-\frac{1}{2}}_{\Curl}}
+\frac{{\rm i}\ka}{b_2}\langle\g_t\ov{\bm w},\g_T{\bm u}
\rangle_{H^{-\frac{1}{2}}_{\Div}\times H^{-\frac{1}{2}}_{\Curl}},\\
\mathcal{F}(({\bm v},{\bm w}))&:=&\frac{-{\rm i}\ka}{b_1\ov{b}_2}\langle {\bm f}_1,\ov{\bm v}
\rangle_{{\bm H}^{-\frac{1}{2}}\times {\bm H}^{\frac{1}{2}}}
-\frac{{\rm i}\ka}{b_2}\langle{\bm f}_2,\g_T\ov{\bm w}
\rangle_{H^{-\frac{1}{2}}_{\Div}\times H^{-\frac{1}{2}}_{\Curl}}.
\enn
We now split the sesquilinear form $\mathcal{A}((\cdot,\cdot),(\cdot,\cdot))$ on $(Q\times X)\times(Q\times X)$ 
into two parts:
\be\label{3.6b}
\mathcal{A}(({\bm u},{\bm H}),({\bm v},{\bm w})):={\bm A}(({\bm u},{\bm H}),({\bm v},{\bm w}))
+{\bm K}(({\bm u},{\bm H}),({\bm v},{\bm w}))
\en
with ${\bm A}((\cdot,\cdot),(\cdot,\cdot))$ and ${\bm K}((\cdot,\cdot),(\cdot,\cdot))$ defined as follows:
\ben
{\bm A}(({\bm u},{\bm H}),({\bm v},{\bm w}))
&:=&\frac{-{\rm i}\ka}{b_1\ov{b}_2}\Int_D(\mathcal{E}({\bm u},\ov{{\bm v}})
+M_1{\bm u}\cdot\ov{\bm v}){\,\rm d}x+\frac{{\rm i}\ka}{\ov{b}_2}\langle\g_t{\bm H},\g_T\ov{\bm v}
\rangle_{H^{-\frac{1}{2}}_{\Div}\times H^{-\frac{1}{2}}_{\Curl}}\\
&&-\Int_{B_R\se\ov{D}}(\curl{\bm H}\cdot\curl\ov{{\bm w}}-\ka^2{\bm H}\cdot\ov{{\bm w}}){\,\rm d}x\\
&&-{\rm i}\ka\langle G_e(\hat{x}\times{\bm H}),\g_T\ov{\bm w}
\rangle_{H^{-\frac{1}{2}}_{\Div}\times H^{-\frac{1}{2}}_{\Curl}}
+\frac{{\rm i}\ka}{b_2}\langle\g_t\ov{\bm w},\g_T{\bm u}
\rangle_{H^{-\frac{1}{2}}_{\Div}\times H^{-\frac{1}{2}}_{\Curl}},\\
{\bm K}(({\bm u},{\bm H}),({\bm v},{\bm w}))
&:=&\frac{{\rm i}\ka}{b_1\ov{b}_2}\Int_D(\rho\om^2+M_1){\bm u}\cdot\ov{\bm v}{\,\rm d}x
\enn
for all $({\bm u},{\bm H}),({\bm v},{\bm w})\in Q\ti X$, where $M_1>0$ is a constant 
chosen such that $M_1>\om^2\|\rho\|_{L_{\infty}}$.
Further, the sesquilinear form ${\bm A}((\cdot,\cdot),(\cdot,\cdot))$ can be written as 
\ben
{\bm A}(({\bm u},{\bm H}), ( {\bm v},{\bm w}))
&=&{\bm A}_2({\bm u},{\bm v})-{\bm A}_1({\bm H},{\bm w})
+\frac{{\rm i}\ka}{\ov{b}_2}\langle\g_t{\bm H},\g_T\ov{\bm v}
\rangle_{H^{-\frac{1}{2}}_{\Div}\times H^{-\frac{1}{2}}_{\Curl}}\\
&&+\frac{{\rm i}\ka}{b_2}\langle\g_t\ov{\bm w},\g_T{\bm u}
\rangle_{H^{-\frac{1}{2}}_{\Div}\times H^{-\frac{1}{2}}_{\Curl}}
\enn
for all $({\bm u},{\bm H}),({\bm v},{\bm w})\in Q\times X$, where
\ben
&&{\bm A}_1({\bm H},{\bm w}):=\Int_{B_R\se\ov{D}}(\curl{\bm H}\cdot\curl\ov{\bm w}
-\ka^2{\bm H}\cdot\ov{\bm w}){\,\rm d}x+{\rm i}\ka\langle G_e(\nu\times{\bm H}),\g_T\ov{\bm w}
\rangle_{H^{-\frac{1}{2}}_{\Div}\times H^{-\frac{1}{2}}_{\Curl}},\\
&&{\bm A}_2({\bm u},{\bm v}):=\frac{-{\rm i}\ka}{b_1\ov{b}_2}\Int_D(\mathcal{E}({\bm u},\ov{{\bm v}})
 +M_1{\bm u}\cdot\ov{\bm v}){\,\rm d}x.
\enn
Note that ${\bm A}_1(\cdot,\cdot)$ corresponds to the magnetic field ${\bm H}$ 
and ${\bm A}_2(\cdot,\cdot)$ corresponds to the elastic field ${\bm u}$. 
It is easy to see that ${\bm A}_2(\cdot,\cdot)$ is coercive in ${\bm H}^1(D)\times{\bm H}^1(D)$. 
However, it is difficult to directly analyze ${\bm A}_1(\cdot,\cdot)$ in 
$H(\curl,B_R\se\ov{D})\times H(\curl,B_R\se\ov{D})$ since the imbedding map of 
$H(\curl,B_R\se\ov{D})\hookrightarrow{\bm L}^2 (B_R\se\ov{D})$ is not compact. 
To overcome this difficulty, we introduce a Helmholtz-type decomposition for ${\bm A}_1(\cdot,\cdot)$. 
To this end, define the scalar space
\ben
S:= \left\{\psi \in H^1(B_R\se\ov{D}):\Int_{S_R}\psi{\,\rm d}s=0\right\},
\enn
which is clearly a closed linear subspace of $H^1(B_R\se\ov{D})$ and thus a Hilbert space.
Then the sesquilinear form ${\bm A}_1(\cdot,\cdot)$ can be rewritten on $S$ as
\be\no
{\bm A}_1(\na\phi,\na\psi):=a(\phi,\psi)+b(\phi,\psi)\quad{\rm for\;all\;}\;\phi,\;\psi\in S,
\en
where $a(\cdot,\cdot)$ and $b(\cdot,\cdot)$ are defined as
\ben
a(\phi,\psi)&:=&-\ka^2\langle\na\phi,\na\ov{\psi}\rangle
+\ka^2\langle\wi{G}_e(\hat{x}\times\na\phi),\na_{S_R}\ov{\psi}
\rangle_{H^{-\frac{1}{2}}_{\Div}\times H^{-\frac{1}{2}}_{\Curl}},\\
b(\phi,\psi)&:=&{\rm i}\ka\langle(G_e+{\rm i}\ka\wi{G}_e)(\hat{x}\times\na\phi),\na_{S_R} \ov{\psi}
\rangle_{H^{-\frac{1}{2}}_{\Div}\times H^{-\frac{1}{2}}_{\Curl}}.
\enn
We have the following result which was proved in \cite{P03}.

\begin{lemma}\label{lem3.4}  
$a(\cdot,\cdot)$ is bounded and elliptic on $S\times S$, and there exists a compact operator $K_1$ on $S$ 
such that $b(\phi,\psi)=a(K_1\phi,\psi)$ for all $\phi,\psi\in S$. Further, $I+K_1$ is an isomorphism on $S$.
\end{lemma}

In order to analyze ${\bm A}_1(\cdot,\cdot)$ on $X\times X$, we introduce the following subspace of $X$:
\be\no
X_0&:=&\{{\bm H}\in X:-\ka^2\langle{\bm H},\na\ov{\psi}\rangle
+{\rm i}\ka\langle G_e(\hat{x}\times{\bm H}),\na_{S_R}\ov{\psi}
\rangle_{H^{-\frac{1}{2}}_{\Div}\times H^{-\frac{1}{2}}_{\Curl}}=0\;\;\forall\psi\in S\}\\ \no
&=&\{{\bm H}\in X:\na\cdot{\bm H}=0\;\;{\rm in}\;B_R\se\ov{D},\;\;
-\ka^2\hat{x}\cdot{\bm H}={\rm i}\ka\na_{S_R}\cdot G_e(\hat{x}\times{\bm H})\;{\rm on}\;S_R\;\;\;\\ \label{3.6cc}
&& {\rm and}\;\;\nu\cdot{\bm H}=0\;\;{\rm on}\;\;\pa D\}.
\en
We then have the following Helmholtz-type decomposition for $X$.

\begin{lemma}\label{lem3.5}
$\na S$ and $X_0$ are closed linear subspaces of $X$, and $X=X_0\bigoplus\na S$ is the direct sum 
of $\na S$ and $X_0$. Further, there exist constants $c_1,c_2>0$ such that
\be\label{3.6aa}
c_1\|{\bm w}+\na\phi\|^2_{X}\le\|{\bm w}\|^2_{X}+\|\na\phi\|^2_{X}\le c_2\|{\bm w}+\na\phi\|^2_{X}
\en
for all ${\bm w}\in X_0$ and $\phi\in S$.
\end{lemma}

\begin{proof}
The closeness of $\na S$ follows from the property that $\curl\na\psi=0$ for $\psi\in X$ and 
the boundedness of the differential operator $\na$ from $H^1(\cdot)$ into $L^2(\cdot)$. 
For a fixed $\psi\in S$, the linear functionals ${\bm H}_0\to\langle{\bm H}_0,\na\ov{\psi}\rangle$ 
and ${\bm H}_0\to\left\langle G_e(\hat{x}\times{\bm H}_0),\na_{S_R}\ov{\psi}\right
\rangle_{H^{-\frac{1}{2}}_{\Div}\times H^{-\frac{1}{2}}_{\Curl}}$ are bounded on $X$, 
yielding that $X_0$ is closed.
	
Given ${\bm H}\in X$, we now construct a function $\phi\in S$ such that
\be\label{3.6aaa}
{\bm A}_1(\na\phi,\na\psi)={\bm A}_1({\bm H},\na\psi)\quad {\rm for\;all\;}\psi\in S.
\en
From Lemma \ref{lem3.4} it follows that such $\phi$ is well defined and satisfies that 
\ben
\|\na\phi\|_{{\bm L}^2(B_R\se\ov{D})}\le c\|{\bm H}\|_{H(\curl,B_R\se\ov{D})}
\enn
for some constant $c>0$. Let ${\bm w}={\bm H}-\na\phi$. Then, and by (\ref{3.6aaa}) and 
the definition of ${\bm A}_1(\cdot,\cdot)$ we deduce that ${\bm w}\in X_0$.
It remains to show that the intersection $\na S\cap X_0$ contains only a trivial element. 
In fact, if there exists $\phi\in S$ such that $\na\phi\in X_0$, then 
\ben
{\bm A}_1(\na\phi,\na\psi)=0\quad {\rm for\;all\;}\psi\in S,
\enn
implying that $\phi=0$.
	
Finally, the inequality (\ref{3.6aa}) follows from the boundedness of the projection operators 
$X\hookrightarrow\na S$ and $X\hookrightarrow X_0$.
\end{proof}

\begin{lemma}\label{lem3.6}
$X_0 $ is compactly imbedded in ${\bm L}^2(B_R\se\ov{D})$.
\end{lemma}

\begin{proof}
Since $X_0$ is a Hilbert space, it is enough to show that ${\bm u}_j\to 0$ in ${\bm L}^2(B_R\se\ov{D})$ 
as $j\to\infty$ if $\{{\bm u}_j\}_{j\in\N}\subset X_0$ and
${\bm u}_j\rightharpoonup 0$ in the weak sense as $j\to\infty$.
	
For each $j\in\N$, define ${\bm v}_j\in H_{loc}(\curl,\R^3\se\ov{B_R})$ which satisfies that 
\ben
\left\{
\begin{array}{ll}
\curl\curl {\bm v}_j-\ka^2{\bm v}_j=0 &\qquad {\rm in}\;\;\mathbb{R}^3\setminus\ov{B_R},\\[1mm]
\hat{x}\times{\bm v}_j=\hat{x}\times{\bm u}_j &\qquad {\rm on}\;\;S_R,\\[1mm]
\ds\lim_{r\to\infty}r(\hat{x}\times\curl{\bm v}_j+{\rm i}\ka{\bm v}_j)=0& \qquad r=|x|.
\end{array}
\right.
\enn
For each $j\in\N$, define 
\ben
{\bm u}^e_j=\begin{cases}
{\bm u}_j  &\qquad{\rm in\;}\;\;B_R\setminus\ov{D},\\
{\bm v}_j  &\qquad{\rm in\;}\;\;\mathbb{R}^3\setminus\ov{B_R}. \\
\end{cases}
\enn
Then it is clear that ${\bm u}^e_j$ is the extension of ${\bm u}_j$ in the sense of $H(\curl,\cdot)$.
	
Recalling the definition of the space $X_0$, one has 
$-\ka^2\hat{x}\cdot{\bm u}_j={\rm i}\ka\na_{S_R}\cdot G_e(\hat{ x}\times{\bm u}_j)$ on $S_R$, 
which, combined with the definition of $G_e$ and the Maxwell equation for ${\bm v}_j$, gives
\be\no
\hat{x}\cdot{\bm u}_j&=&-\frac{\rm i}{\ka}\na_{S_R}\cdot G_e(\hat{ x}\times{\bm u}_j)
=-\frac{\rm i}{\ka}\na_{S_R}\cdot\left(\frac{1}{\rm i\ka}\hat{ x}\times\curl{\bm v}_j\right)\\
&=&-\frac{1}{\ka^2}\na_{S_R}\cdot(\hat{x}\times\curl{\bm v}_j)
=\frac{1}{\ka^2}\hat{ x}\cdot\curl\curl{\bm v}_j
=\hat{ x}\cdot{\bm v}_j\quad {\rm on\;\;} S_R.\;\; \label{3.8aa}
\en
Noting that $\na\cdot{\bm u}_j=0$ in $B_R\se\ov{D}$ and $\na\cdot{\bm v}_j=0$ in $\R^3\se\ov{B_R}$, 
we conclude from (\ref{3.8aa}) that the extended function ${\bm u}^e_j$ satisfies
\ben
\dive{\bm u}^e_j=0\quad {\rm in\;}\;\;\R^3\se\ov{D}\qquad{\rm and}\qquad
\nu\cdot{\bm u}^e_j=0\quad{\rm on\;}\;\;\pa D.
\enn
Then it follows from Theorem 3.50 of \cite{P03} that ${\bm u}_j^e\in H^{1/2+s}_{loc}(\R^3\se\ov{D})$ 
for some $s\geq0$. By the compactness of the imbedding $H^{1/2}(B_R\se\ov{D})\hookrightarrow L^2(B_R\se\ov{D})$,
there is a subsequence of $\{{\bm u}_j\}$ converging to $0$ in ${\bm L}^2(B_R\se\ov{D})$. 
This completes the proof.
\end{proof}

We are now ready to analyze the sesquilinear form ${\bm A}_1(\cdot,\cdot)$ on $X_0$. 
First, for ${\bm H}$, ${\bm w}\in X$, by Lemma \ref{lem3.5} there exist
${\bm H}_0, {\bm w}_0\in X_0$ and $\phi,\psi\in S$ such that ${\bm H}={\bm H}_0+\na \phi$ 
and ${\bm w}={\bm w}_0+\na\psi$. Thus, by the definition of $X_0$ one has
\be\label{3.9aa}
{\bm A}_1({\bm H},{\bm w})={\bm A}_1({\bm H}_0,{\bm w}_0)+{\bm A}_1(\na \phi,{\bm w}_0)
+{\bm A}_1(\na\phi,\na\psi).
\en
We split ${\bm A}_1(\cdot,\cdot)$ into two parts:
\ben
{\bm A}_1({\bm H}_0,{\bm w}_0):=a_0({\bm H}_0,{\bm w}_0)
+b_0({\bm H}_0,{\bm w}_0)\qquad\forall\;{\bm H}_0,\;{\bm w}_0\in {X}_0,
\enn
where
\ben
a_0({\bm H}_0,{\bm w}_0)&:=&(\curl{\bm H}_0,\curl\ov{\bm w}_0)
+({\bm H}_0,\ov{\bm w}_0)+{\rm i}\ka\langle G_e^2(\hat{x}\times{\bm H}_0),\g_T\ov{\bm w}_0 
\rangle_{H^{-\frac{1}{2}}_{\Div}\times H^{-\frac{1}{2}}_{\Curl}}, \\
b_0({\bm H}_0,{\bm w}_0)&:=&-(\ka^2+1)({\bm H}_0,\ov{\bm w}_0)
+{\rm i}\ka\langle G_e^1(\hat{ x}\ti{\bm H}_0),\g_T\ov{\bm w}_0
\rangle_{H^{-\frac{1}{2}}_{\Div}\times H^{-\frac{1}{2}}_{\Curl}}.
\enn
Similar to Lemma \ref{lem3.4}, we have the following result for $a_0(\cdot,\cdot)$ 
and $b_0(\cdot,\cdot)$ on $X_0$.

\begin{lemma}\label{lem3.7}
$a_0(\cdot,\cdot)$ is coercive on $X_0\times X_0$, and there exists a compact operator $K_2$ 
on $X_0$ such that $b_0({\bm H}_0,{\bm w}_0)=a_0(K_2{\bm H}_0,{\bm w}_0)$ for all 
${\bm H}_0,{\bm w}_0\in X_0$. Further, $I+K_2$ is an isomorphism on $X_0$.
\end{lemma}

\begin{proof}
The coerciveness of $a_0(\cdot.\cdot)$ follows easily from the property of $G^2_e$ 
(see (c.2) in Lemma \ref{lem3.3}).
By the Cauchy-Schwartz inequality, it is easy to see that, for each fixed ${\bm H}_0\in X_0$,
$b_0({\bm H}_0,\cdot)$ defines a bounded linear functional on $X_0$. Thus, and by the Lax-Milgram theorem
and the coerciveness of $a_0(\cdot.\cdot)$ on $X_0\times X_0$,
there is an operator $K_2$ on $X_0$ such that $b_0({\bm H}_0,{\bm w}_0)=a_0(K_2{\bm H}_0,{\bm w}_0)$.
The compactness of $K_2$ follows easily from the property of $G^1_e$ (see Lemma \ref{lem3.3}) 
and the compact imbedding $X_0\hookrightarrow{\bm L}^2(B_R\se\ov{D})$ (see Lemma \ref{lem3.6}). 

We now prove that $I+K_2$ is an isomorphism on $X_0$. By the Risze-Fredholm theory, it is enough 
to show that $I+K_2$ is injective. Let $(I+K_2){\bm w}=0$ with ${\bm w}\in X_0$. 
Then ${\bm w}$ satisfies
\ben\no
{\bm A}_1({\bm w},\psi)=a_0({\bm w}+ K_2{\bm w},\psi)=0 \quad {\rm for\;all\;\psi\in X_0}.
\enn
By the definition of $X_0$, we know that ${\bm A}_1({\bm w},\na\phi)=0$ for all $\phi\in S$.
This, combined with the Helmholtz-type decomposition for $\Psi\in X$, yields
\ben
{\bm A}_1({\bm w},\Psi)=0\quad{\rm for\;all\;}\Psi\in X.
\enn
Therefore, ${\bm w}$ satisfies the boundary value problem
\ben
\left\{
\begin{array}{ll}
\ds\curl\curl {\bm w}-\ka^2{\bm w}= 0& \qquad {\rm in\;}\;\; B_R\setminus\ov{D},\\[1mm]
\ds\nu\times\curl{\bm w}=0 &\qquad {\rm on\;}\;\;\pa D,\\[1mm]
\ds\frac{1}{\rm i\ka}\hat{x}\times\curl{\bm w}=G_e(\hat{x}\times{\bm w})&\qquad{\rm on}\;\;\;S_R
\end{array}
\right.
\enn
in the distribution sense. 
By the third equation in the above problem it is seen that ${\bm w}$ can be extended into 
$\R^3\se\ov{B_R}$ by considering the exterior problem
\ben
\left\{
\begin{array}{ll}
\ds\curl\curl {\bm v}-\ka^2{\bm v}=0 &\qquad {\rm in\,}\;\;\mathbb{R}^3\setminus\ov{B_R},\\[1mm]
\ds\hat{x}\times{\bm v}=\hat{x}\times{\bm w}&\qquad {\rm on\,}\;\;S_R,\\[1mm]
\ds\lim_{r\to\infty}r(\hat{x}\ti\curl{\bm v}+ {\rm i}\ka{\bm v})=0,& \qquad r=|x|.
\end{array}
\right.
\enn
By the definition of $G_e$ it follows that $\hat{x}\times\curl{\bm w}=\hat{x}\times\curl{\bm v}$ on $S_R$. 
Hence, the function ${\bm w}^e$, which is defined by ${\bm w}$ in $B_R\se\ov{D}$ and by ${\bm v}$ 
in $\R^3\se\ov{B_R}$, satisfies the exterior problem
\ben
\left\{
\begin{array}{ll}
\ds\curl\curl{\bm w}^e-\ka^2{\bm w}^e=0 &\qquad {\rm in\;}\;\;\mathbb{R}^3\setminus\ov{D},\\[1mm]
\ds\nu\times\curl{\bm w}^e=0\ &\qquad {\rm on\;}\;\;\pa D,\\[1mm]
\ds\lim_{r\to\infty}r(\hat{x}\times\curl{\bm w}^e+{\rm i\ka}{\bm w}^e)=0,&\qquad r=|x|.
\end{array}
\right.
\enn
Using Green's formula for ${\bm w}^e$, one has
\ben
\Int_{S_R}\nu\times\ov{{\bm w}^e}\cdot\frac{1}{\rm i\ka}\curl{\bm w}^e\;{\rm d}s=0.
\enn
This, together with the Rellich lemma and the unique continuation principle, implies 
${\bm w}^e\equiv0$ in $\R^3\se\ov{D}$, and thus $w=0$. Therefore, $I+K_2$ is injective on $X_0$, 
which ends the proof.	
\end{proof}

Based on the above analysis and Lemmas \ref{lem3.4}, \ref{lem3.5}, \ref{lem3.6} and \ref{lem3.7}, 
we can split ${\bm A}(({\bm u},{\bm H}),({\bm v},{\bm w}))$ as follows:
\ben
{\bm A}(({\bm u},{\bm H}),({\bm v},{\bm w})):=\wi{\bm A}(({\bm u},{\bm H}),({\bm v},{\bm w}))
+{\bm K}_3(({\bm u},{\bm H}),({\bm v},{\bm w}))
\enn
for all ${\bm u},{\bm v}\in Q$ and ${\bm H},{\bm w}\in X$ with ${\bm H}={\bm H}_0+\na\phi$ 
and ${\bm w}={\bm w}_0+\na\psi$, ${\bm H}_0,{\bm w}_0\in X_0$ and $\phi,\psi\in S$,
where $\wi{\bm A}(\cdot,\cdot)$ and ${\bm K}_3(\cdot,\cdot)$ are defined as
\ben
\wi{\bm A}(({\bm u},{\bm H}),({\bm v},{\bm w}))
&:=&{\bm A}_2({\bm u},{\bm v})-a(\phi,\psi)-a_0({\bm H}_0,{\bm w}_0)-{\bm A}_1(\na\phi,{\bm w}_0) \\
&&+\frac{{\rm i}\ka}{\ov{b}_2}\langle\g_t\na\phi,\g_T\ov{\bm v}
\rangle_{H^{-\frac{1}{2}}_{\Div}\times H^{-\frac{1}{2}}_{\Curl}}
+\frac{{\rm i}\ka}{b_2}\langle\g_t\na\ov{\psi},\g_T{\bm u}
\rangle_{H^{-\frac{1}{2}}_{\Div}\times H^{-\frac{1}{2}}_{\Curl}}, \\
{\bm K}_3(({\bm u},{\bm H}),({\bm v},{\bm w}))
&:=&-b(\phi,\psi)-b_0({\bm H}_0,{\bm w}_0)+\frac{{\rm i}\ka}{\ov{b}_2}
\langle\g_t{\bm H}_0,\g_T\ov{\bm v}\rangle_{H^{-\frac{1}{2}}_{\Div}\times H^{-\frac{1}{2}}_{\Curl}}\\
&&+\frac{{\rm i}\ka}{b_2}\langle\g_t\ov{\bm w}_0,\g_T{\bm u}
\rangle_{H^{-\frac{1}{2}}_{\Div}\times H^{-\frac{1}{2}}_{\Curl}}.
\enn
Further, define the sesquilinear form
\ben\no
\wi{\bm K}(({\bm u},{\bm H}),({\bm v},{\bm w})):={\bm K}(({\bm u},{\bm H}),({\bm v},{\bm w}))
+{\bm K}_3(({\bm u},{\bm H}),({\bm v},{\bm w})).
\enn
Then (\ref{3.4a}) can be reduced to the problem: find $({\bm u},{\bm H})\in Q\times X$ such that
\be\label{3.11b}
\wi{\bm A}(({\bm u},{\bm H}),({\bm v},{\bm w}))+\wi{\bm K}(({\bm u},{\bm H}),({\bm v},{\bm w})) 
=\mathcal{F}(({\bm v},{\bm w}))
\en
for all $({\bm v},{\bm w})\in Q\times X$.

Let $\wi{\bm A},\wi{\bm K}:Q\times X\mapsto (Q\ti X)^\prime$ be the linear, bounded operators 
induced by the corresponding sesquilinear forms $\wi{\bm A}(\cdot,\cdot)$, $\wi{\bm K}(\cdot,\cdot)$,
respectively, with using the Riesz representation lemma in Hilbert spaces. 
Then we have the following result.

\begin{theorem}\label{thm3.8}
If $\I(b_1\ov{b}_2)<0$, then $\wi{\bm A}+\wi{\bm K}$ is of Fredholm type with index $0$.
\end{theorem}

\begin{proof}
By the compact imbedding ${\bm H}^1(D)\mapsto{\bm L}^2(D)$ we deduce that ${\bm K}(\cdot,\cdot)$ 
is a compact form on ${\bm H}^1(\cdot)\times{\bm H}^1(\cdot)$.
By Lemma \ref{lem3.3} it is known that $b(\cdot,\cdot)$ is a compact form on $S\times S$, 
and by Lemmas \ref{lem3.6} and \ref{lem3.3} it is also known that $b_0(\cdot,\cdot)$ is a compact form
on $X_0\times X_0$. Further, by a similar argument as in deriving (\ref{3.8aa}) (see also \cite{P03})
it follows that, if ${\bm w}_0\in X_0$ then ${\bm w}_0|_{\pa D}\in{\bm H}^{\frac{1}{2}}(\pa D)$.
This, combined with the compact imbedding ${\bm H}^{\frac{1}{2}}(\pa D)\mapsto{\bm L}^2(\pa D)$, 
gives that ${\bm K}_3(\cdot,\cdot)$ is a compact form on $(Q\times X)\times(Q\times X)$. 
Thus, the operator $\wi{\bm K}$ is compact from $Q\times X$ into $(Q\times X)^\prime$.
It remains to show that $\wi{\bm A}$ is an isomorphism from $Q\times X$ into $(Q\times X)^\prime$.
	
Since $\I (b_1\ov{b}_2)<0$, we obtain by using Korn's inequality that
\be\label{3.5a}
\Rt\left[{\bm A}_2({\bm u},{\bm u})\right]\ge C\|{\bm u}\|_Q^2\quad{\rm for\;all\;}{\bm u}\in Q
\en
for some constant $C>0$. Further, by Lemmas \ref{lem3.3} and \ref{lem3.4} it can be concluded that
\be\label{3.6a}
&&-\Rt\left[a(\phi,\phi)\right]\ge C\|\na\phi\|^2_{X}\quad{\rm for\;all\;}\phi\in S,\\ \label{3.8a}
&&\;\;\;\Rt\left[a_0({\bm H}_0,{\bm H}_0)\right]\ge C\|{\bm H}_0\|^2_X\quad
{\rm for\;all\;}{\bm H}_0\in X_0
\en
for some constant $C>0$.	
Recalling that
\ben
\frac{{\rm i}\ka c}{b_2}+\frac{{\rm i}\ka\ov{c}}{\ov{b_2}}=2{\rm i}\ka\Rt(b_2\ov{c})\frac{1}{|b_2|^2},
\enn
we immediately have
\be\label{3.15a}
\Rt\left[\frac{{\rm i}\ka}{\ov{b}_2}\langle\g_t\na\phi,\g_T\ov{\bm u}
\rangle_{H^{-\frac{1}{2}}_{\Div}\times H^{-\frac{1}{2}}_{\Curl}}
+\frac{{\rm i}\ka}{b_2}\langle\g_t\na\ov{\phi},\g_T{\bm u}
\rangle_{H^{-\frac{1}{2}}_{\Div}\times H^{-\frac{1}{2}}_{\Curl}}\right]=0.
\en
Therefore, by (\ref{3.5a})-(\ref{3.6a}) and (\ref{3.15a}) we obtain that the real part of
\ben
\wi{\bm A}_1(({\bm u},\phi),({\bm v},\psi))&:=&{\bm A}_2({\bm u},{\bm v})-a(\phi,\psi) \\
&& +\frac{{\rm i}\ka}{\ov{b}_2}\langle\g_t\na\phi,\g_T\ov{\bm v}
\rangle_{H^{-\frac{1}{2}}_{\Div}\times H^{-\frac{1}{2}}_{\Curl}}
+\frac{{\rm i}\ka}{b_2}\langle\g_t\na\ov{\psi},\g_T{\bm u}
\rangle_{H^{-\frac{1}{2}}_{\Div}\times H^{-\frac{1}{2}}_{\Curl}}
\enn
is coercive on $(Q\times S)\times(Q\times S)$. Thus, and by the Lax-Milgram lemma, 
for each bounded functional $(f_1,f_2)\in Q'\times(\na S)'$ there exists a unique element 
$(\wi{\bm u},\wi{\phi})\in (Q\times S)$ satisfying that
\be\label{3.17a}
\wi{\bm A}_1((\wi{\bm u},\wi{\phi}),({\bm v},\psi))=f_1({\bm v})
+f_2(\na\psi)\quad{\rm for\;all\;}({\bm v},\psi)\in Q\times S
\en
and the estimate
\be\label{3.18a}
\|\wi{\bm u}\|_Q+\|\na\wi{\phi}\|_X\leq C(\|f_1\|_{Q'} +\| f_2\|_{(\na S)'})
\en
for some constant $C>0$.
	
Moreover, by (\ref{3.8a}), the boundedness of ${\bm A}_1(\cdot,\cdot)$ and the Lax-Milgram lemma,
for each $f_3\in (X_0)'$ there exists a unique element $\wi{\bm H}_0\in X_0$ satisfying that
\be\label{3.19a}
-a_0(\wi{\bm H}_0,{\bm w}_0)={\bm A}_1(\na\wi{\phi},{\bm w}_0)
+f_3({\bm w}_0)\quad {\rm for\;all\;}{\bm w}_0\in X_0
\en
and the estimate
\be\label{3.20a}
\|\wi{\bm H}_0\|_{X}\le C(\|\na\wi{\phi}\|_X+\|f_3\|_{(X_0)'})
\le C(\|f_1\|_{Q'}+\|f_2\|_{(\na S)'}+\|f_3\|_{(X_0)'}),
\en
where $C>0$ is a constant.
Combining (\ref{3.17a})-(\ref{3.20a}) implies that $\wi{\bm A}$ is an isomorphism 
from $Q\times X$ to $(Q\times X)^{'}$. The proof is thus compete.
\end{proof}

Using Theorem \ref{thm3.8}, we can easily obtain the following well-posedness result for the problem (\ref{3.3a}).

\begin{theorem}\label{thm3.9}
Let $\om\notin\mathcal{P}(\om)$. If $\Rt(b_1\ov{b}_2)=0$ and $\I(b_1\ov{b}_2)<0$, 
then the problem (\ref{3.3a}) has a unique solution $({\bm u},{\bm H})\in Q\times X$ 
satisfying the estimate 
\ben
\|{\bm u}\|_Q+\|{\bm H}\|_X\le C(\|{\bm f}_1\|_{{\bm H}^{-\frac{1}{2}}(\pa D)}
+\|{\bm f}_2\|_{H^{-\frac{1}{2}}_{\Div}(\pa D)}),
\enn
where $C>0$ is a constant independent of the choice of ${\bm f}_1$ and ${\bm f}_2$.
\end{theorem}

\section{Uniqueness of the inverse problem}\label{sec4}
\setcounter{equation}{0}

In this section, based on the analysis for the forward scattering problem (\ref{2.1a})-(\ref{2.5a}), 
we investigate the inverse problem of determining the elastic body $D$ by the electromagnetic far-field 
measurements. We shall show that the shape and location of the elastic body can be uniquely recovered  
by the magnetic or electric far-field pattern corresponding to incident plane waves with all incident 
directions and polarizations. Motivated by our previous work in \cite{JBH18} 
for the reduced wave equation and the Maxwell equations, our method is based on a coupled system of PDEs 
constructed in a sufficiently small domain as well as the uniform a priori estimate in ${\bm H}^1(\cdot)$ 
for the elastic field.

\subsection{A coupled system of PDEs}

In order to study the inverse problem, we introduce the following boundary value problem in a bounded, 
simply connected domain $\Om$ with a Lipschitz continuous boundary $\pa\Om$:
\be\label{4.1a}
\left\{\begin{array}{ll}
\ds\curl\curl{\bm H}+{\bm H}=\xi_1& \quad {\rm in\;}\;\;\Om,\\[1mm]
\ds\na\cdot(\hat{\mathcal{C}}:\na{\bm u})-{\bm u}=\xi_2 & \quad{\rm in\;}\;\;\Om,\\[1mm]
\ds T{\bm u}-\hat{b}_1\nu\times{\bm H}={\bm h}_1 & \quad {\rm on\;}\;\;\pa\Om,\\[1mm]
\ds\nu\times\curl{\bm H} +\frac{{\rm i}\ka}{\hat{b}_2}\nu\times{\bm u}={\bm h}_2&\quad{\rm on\;}\;\;\pa\Om,
\end{array}
\right.
\en
where $\ds\hat{\mathcal{C}}:=(\hat{C}_{ijkl}(x))_{i,j,k,l=1}^3$ with $\ds\hat{C}_{ijkl}\in L_\infty(\Om)$ 
($i,j,k,l=1,2,3$), $\ds\hat{b}_1,\hat{b}_2\in\C$ with $\ds\hat{b}_1\hat{b}_2\not=0$, 
$\ds\xi_1,\xi_2\in{\bm L}^2(\Om)$, $\ds{\bm h}_1\in{\bm H}^{-1/2}(\pa\Om)$,
$\ds{\bm h}_2\in H^{-1/2}_{\Div}(\pa\Om)$ and $\hat{\mathcal{C}}$ satisfies the symmetry condition and the
Legendre elliptic condition (see Subsection \ref{sec2.2}).

\begin{lemma}\label{lem4.1}
If $\I(\hat{b}_1\ov{\hat{b}}_2)<0$, then the problem $(\ref{4.1a})$ has a unique solution 
$({\bm H},{\bm u})\in H(\curl,\Om)\times{\bm H}^1(\Om)$ such that
\ben
&&\|{\bm H}\|_{H(\curl,\Om)}+\|{\bm u} \|_{{\bm H}^1(\Om)}\\
&&\qquad\qquad\le C\left[\|\xi_1\|_{{\bm L}^2(\Om)}+\|\xi_2\|_{{\bm L}^2(\Om)}
+\|{\bm h}_1\|_{{\bm H}^{-1/2}(\pa\Om)}+\|{\bm h}_2\|_{H^{-1/2}_{\Div}(\pa\Om)}\right],
\enn
where $C>0$ is a constant independent of $\xi_1,\xi_2,{\bm h}_1$ and ${\bm h}_2$.
\end{lemma}

\begin{proof}
By using Green's formula, the problem (\ref{4.1a}) can be reformulated as the variational problem: 
find $({\bm H},{\bm u})\in H(\curl,\Om)\times{\bm H}^1(\Om)$ such that
\be\label{4.2b}
\hat{\mathcal{A}}(({\bm H},{\bm u}),({\bm w},{\bm v}))=\hat{\mathcal{F}}(({\bm w},{\bm v}))
\quad{\rm for\;}\;\;({\bm w},{\bm v})\in H(\curl,\Om)\times{\bm H}^1(\Om),
\en
where
\ben
\hat{\mathcal{A}}(({\bm H},{\bm u}),({\bm w},{\bm v}))
&:=&\Int_{\Om}(\curl{\bm H}\cdot\curl\ov{\bm w}+{\bm H}\cdot\ov{\bm w}){\,\rm d}x 
+ \frac{-{\rm i}\ka}{\hat{b}_1\ov{\hat{b}}_2}\Int_{\Om}(\mathcal{E}({\bm u},\ov{\bm v})
+{\bm u}\cdot\ov{\bm v})){\,\rm d}x\\
&&+\frac{{\rm i}\ka}{\hat{b}_2}\langle\g_t\ov{\bm w},\g_T{\bm u}
\rangle_{H^{-\frac{1}{2}}_{\Div}\times H^{-\frac{1}{2}}_{\Curl}}
-\frac{-{\rm i}\ka}{\ov{\hat{b}}_2}\langle\g_t{\bm H},\g_T\ov{\bm v}
\rangle_{H^{-\frac{1}{2}}_{\Div}\times H^{-\frac{1}{2}}_{\Curl}},\\
\hat{\mathcal{F}}(({\bm w},{\bm v}))
&:=&\frac{-{\rm i}\ka}{\hat{b}_1\ov{\hat{b}}_2}\langle{\bm h}_1,\ov{\bm v}
\rangle_{{\bm H}^{-\frac{1}{2}}\times{\bm H}^{\frac{1}{2}}}
-\langle{\bm h}_2,\g_T\ov{\bm w}\rangle_{H^{-\frac{1}{2}}_{\Div}\times H^{-\frac{1}{2}}_{\Curl}}\\
&&+\Int_{\Om}(\xi_1\cdot\ov{\bm w}+\frac{{\rm i}\ka}{\hat{b}_1\ov{\hat{b}}_2}\xi_2\cdot\ov{\bm v}){\,\rm d}x.
\enn
By the Cauchy-Schwarz inequality and the trace theorem, there exists a constant $C>0$ such that
\ben
&&|\hat{\mathcal{A}}(({\bm H},{\bm u}),({\bm w},{\bm v}))|
\le C\|({\bm H},{\bm u})\|_{H(\curl,\Om)\times{\bm H}^1(\Om)}
\|({\bm w},{\bm v})\|_{H(\curl,\Om)\times{\bm H}^1(\Om)},\\
&&|\hat{\mathcal{F}}(({\bm w},{\bm v}))|\le C\left[\|\xi_1\|_{{\bm L}^2(\Om)}
+\|\xi_2\|_{{\bm L}^2(\Om)}\right.\\
&&\qquad\qquad\qquad\quad\;\;\;\left.+\|{\bm h}_1\|_{{\bm H}^{-\frac{1}{2}}(\pa\Om)}
+\|{\bm h}_2\|_{H^{-\frac{1}{2}}_{\Div}(\pa\Om)}\right]
\|({\bm w},{\bm v})\|_{H(\curl,\Om)\times{\bm H}^1(\Om)},
\enn
which implies that both $\hat{\mathcal{A}}(\cdot,\cdot)$ and $\hat{\mathcal{F}}(\cdot)$ are bounded
in $H(\curl,\Om)\times{\bm H}^1(\Om)$.
	
Now, let $({\bm w},{\bm v}):=({\bm H},{\bm u})$ in (\ref{4.2b}). Then it follows that 
\ben
\hat{\mathcal{A}}(({\bm H},{\bm u}),({\bm H},{\bm u}))
&=&\Int_{\Om}(\curl{\bm H}\cdot\curl\ov{\bm H}+|{\bm H}|^2
+\frac{-{\rm i}\ka}{\hat{b}_1\ov{\hat{b}}_2}(\mathcal{E}({\bm u},\ov{\bm u})+|{\bm u}|^2)){\,\rm d}x\\
&&+\frac{{\rm i}\ka}{\hat{b}_2}\langle\g_t\ov{\bm H},\g_T{\bm u}
\rangle_{H^{-\frac{1}{2}}_{\Div}\times H^{-\frac{1}{2}}_{\Curl}}
-\frac{-{\rm i}\ka}{\ov{\hat{b}}_2}\langle\g_t{\bm H},\g_T\ov{\bm u}
\rangle_{H^{-\frac{1}{2}}_{\Div}\times H^{-\frac{1}{2}}_{\Curl}}.
\enn
Thus,
\ben
\Rt\left[\hat{\mathcal{A}}(({\bm H},{\bm u}),({\bm H},{\bm u}))\right]
\ge c\left[\|{\bm H}\|^2_{H(\curl,\Om)}+\|{\bm u}\|^2_{{\bm H}^1(\Om)}\right]
\enn
for some constant $c>0$. The required result then follows from the Lax-Milgram lemma.
\end{proof}

\subsection{Uniqueness in recovering the elastic body}

Assume that $D$ and $\wi{D}$ are two elastic bodies corresponding to with the electromagnetic far-field 
patterns (${\bm E}^{s}_{\infty}(\hat{x},d,p)$,${\bm H}^{s}_{\infty}(\hat{x},d,p)$) and
($\wi{\bm E}^{s}_{\infty}(\hat{x},d,p)$,$\wi{\bm H}^{s}_{\infty}(\hat{x},d,p)$), 
generated by the incident plane waves given in (\ref{2.1a}) with the incident direction $d\in\mat{S}^2$ 
and the polarization vector $p\in\R^3$.

\begin{theorem}\label{thm4.2}
If ${\bm H}^{s}_{\infty}(\hat{x},d,p)=\wi{\bm H}^{s}_{\infty}(\hat{x},d,p)$ for all $\hat{x},d\in\mat{S}^2$ 
and $p\in\mat{R}^3$, then $D=\wi{D}$.
\end{theorem}

\begin{proof}
Suppose $D\not=\wi{D}$. Then there would exist $z_*\in\pa D\se\pa\wi{D}$ and a small ball $B$ centered 
at $z_*$ such that
\ben
z_j:=z_*+\frac{\delta}{j}\nu(z_*)\in B, \quad {\rm for\;}\; j=1,2,\cdots
\enn
for small enough $\delta>0$ and $B\cap\ov{\wi{D}}=\emptyset$; see Figure \ref{fig:labe2} 
for the geometric description.
\begin{figure}[htp]
\centering
\includegraphics[scale=0.40]{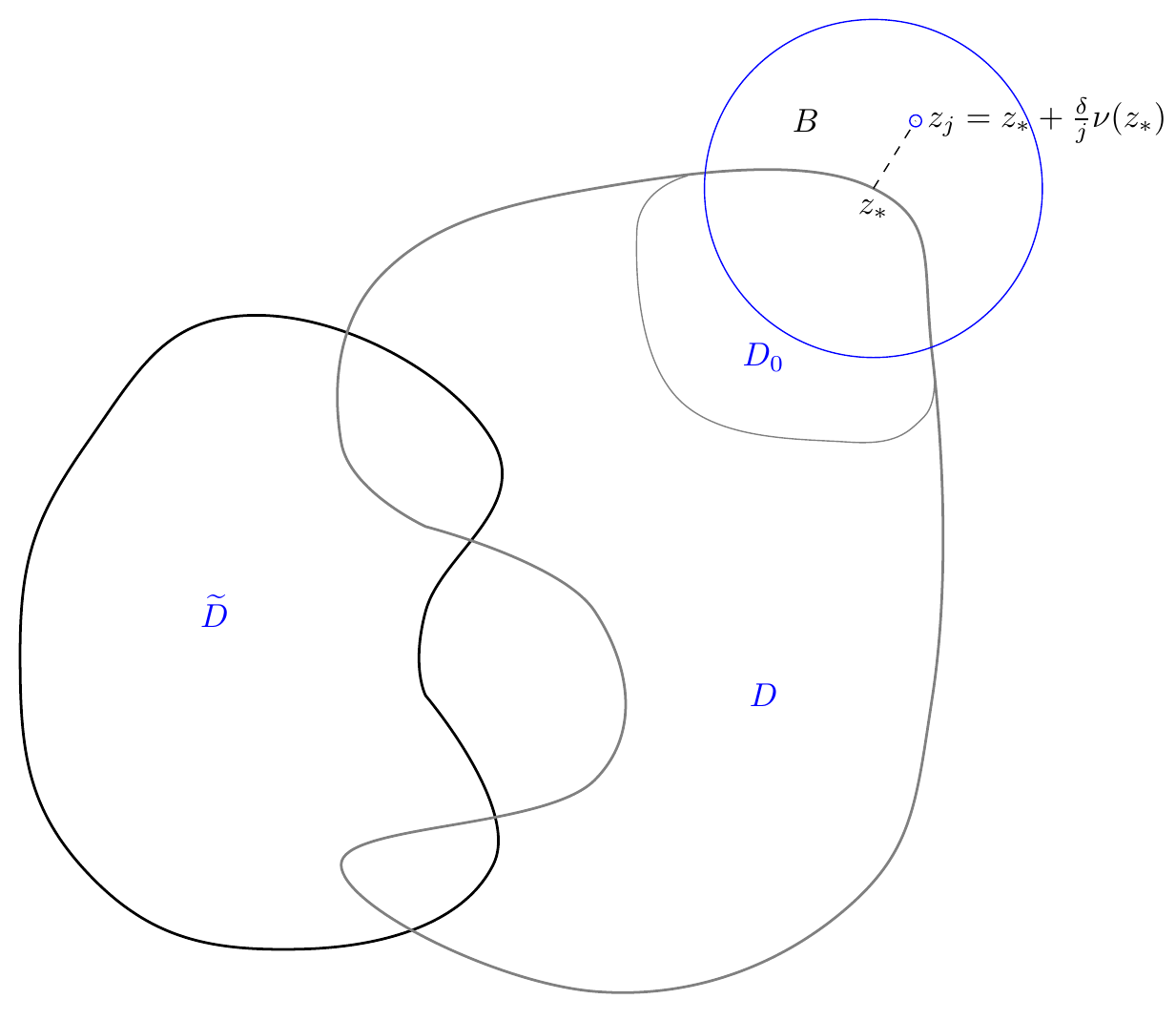}
\caption{Two different elastic bodies}
\vskip -0.8truecm
\hspace{0.2in}\label{fig:labe2}
\end{figure}
	
Consider the scattering problem (\ref{3.2a}) with the boundary data ${\bm f}_1$ and ${\bm f}_2$ 
induced by the electric dipoles
\be\label{4.3b}
&&{\bm H}^i(x,z_j,q)=\curl(q\Phi(x,z_j))/\|\curl(q\Phi(x,z_j))\|_{{\bm L}^2(\pa D)},\\ \label{4.4b}
&&{\bm E}^i(x,z_j,q)=-\frac{1}{\rm i\ka}\curl{\bm H}^i(x,z_j,q)
\en
for $q\in\R^3$, where $\Phi(\cdot,\cdot)$ is the fundamental solution to the three-dimensional 
Helmhlotz equation given by
\ben
\Phi(x,z):=\frac{1}{4\pi}\frac{e^{{\rm i}\ka|x-z| }}{|x-z|},\qquad\; x\not= z.
\enn
By Theorem \ref{thm3.9} we know that, for each $j\in\N$ the problem (\ref{3.2a}) has 
a unique solution $({\bm E}^s(\cdot,z_j,q),{\bm H}^s(\cdot,z_j,q),{\bm u}(\cdot,z_j,q))
\in H_{loc}(\curl,\R^3\se\ov{D})\times H_{loc}(\curl,\R^3\se\ov{D})\times{\bm H}^1(D)$ 
with respect to the elastic body $D$ and 
a unique solution $(\wi{{\bm E}}^s(\cdot,z_j,q),\wi{{\bm H}}^s(\cdot,z_j,q),\wi{{\bm u}}(\cdot,z_j,q))
\in H_{loc}(\curl,\R^3\se\ov{\wi{D}})\times H_{loc}(\curl,\R^3\se\ov{\wi{D}})\times{\bm H}^1(\wi{D})$
with respect to the elastic body $\wi{D}$. Define the total electromagnetic fields as follows:
\ben
&&{\bm E}(\cdot,z_j,q): = {\bm E}^i(x,z_j,q) + {\bm E}^s(\cdot,z_j,q)\qquad {\rm in\;}\R^3\se\ov{D},\\
&&{\bm H}(\cdot,z_j,q): = {\bm H}^i(x,z_j,q) + {\bm H}^s(\cdot,z_j,q)\qquad {\rm in\;}\R^3\se\ov{D},\\
&&\wi{{\bm E}}(\cdot,z_j,q):=\wi{{\bm E}}^i(x,z_j,q)+\wi{{\bm E}}^s(\cdot,z_j,q)
\qquad{\rm in\;}\R^3\se\ov{\wi{D}},\\
&&\wi{{\bm H}}(\cdot,z_j,q):=\wi{{\bm H}}^i(x,z_j,q)+\wi{{\bm H}}^s(\cdot,z_j,q)
\qquad{\rm in\;}\R^3\se\ov{\wi{D}}.
\enn
	
We now prove that the following mixed reciprocity relation holds for the scattering solutions of 
the problem (\ref{3.2a}) associated with the incident plane wave given in (\ref{2.1a}) and 
the electric dipoles given in (\ref{4.3b}) and (\ref{4.4b}):
\be\label{4.5b}
\frac{1}{c_j}4\pi p\cdot{\bm E}^s_{\infty}(-d,z_j,q)=q\cdot{\bm E}^s(z_j,d,p),\\\label{4.6b}
\frac{1}{c_j}4\pi p\cdot\wi{{\bm E}}^s_{\infty}(-d,z_j,q)=q\cdot\wi{{\bm E}}^s(z_j,d,p),
\en
where $\ds c_j:=1/\|\curl(q\Phi(x,z_j))\|_{{\bm L}^2(\pa D)}$, $\ds{\bm E}^s_{\infty}(-d,z_j,q)$ 
and $\ds\wi{{\bm E}}^s_{\infty}(-d,z_j,q)$ are the electric far-field patterns corresponding 
to $D$ and $\wi{D}$, respectively. We only prove (\ref{4.5b}) since (\ref{4.6b}) can be shown similarly.
First, use the vector Gauss divergence theorem and the radiation condition (\ref{2.4a}) to obtain 
that for each $p,q\in\mat{R}^3$,
\ben\label{4.7b}
&&\Int_{\pa D}(\nu\times{\bm E}^i(\cdot,z_j,q)\cdot{\bm H}^i(\cdot,d,p)
+\nu\times{\bm H}^i(\cdot,z_j,q)\cdot{\bm E}^i(\cdot,d,p)){\,\rm d}s=0,\\ \label{4.8b}
&&\Int_{\pa D}(\nu\times{\bm E}^s(\cdot,z_j,q)\cdot{\bm H}^s(\cdot,d,p)
+\nu\times{\bm H}^s(\cdot,z_j,q)\cdot{\bm E}^s(\cdot,d,p)){\,\rm d}s=0.
\enn
Next, by the Stratton-Chu formula (cf. \cite{DR13}) we get
\ben\no
&&4\pi p\cdot{\bm E}^s_{\infty}(-d,z_j,q) 
=\Int_{\pa D}(\nu\times{\bm E}^s_e(\cdot,z_j,q)\cdot{\bm H}^i(\cdot,d,p)
+ \nu\times{\bm H}^s(\cdot,z_j,q)\cdot{\bm E}^i(\cdot,d,p)){\,\rm d}s, \\ \no
&&q\cdot{\bm E}^s(z_j,d,p) 
=\frac{1}{c_j}\Int_{\pa D}(\nu\times{\bm E}^s(\cdot,d,p)\cdot{\bm H}^i(\cdot,z_j,q)
+\nu\times{\bm H}^s(\cdot,d,p)\cdot{\bm E}^i(\cdot,z_j,q)){\,\rm d}s.
\enn
Combining the above four equations with the transmission conditions yields
\ben
&&4\pi p\cdot{\bm E}^s_{\infty}(-d,z_j,q)-c_jq\cdot{\bm E}^s(z_j,d,p)\\
&&=\Int_{\pa D}\left[\nu\times{\bm E}(\cdot,z_j,q)\cdot{\bm H}(\cdot,d,p)
+\nu\times{\bm H}(\cdot,z_j,q)\cdot{\bm E}(\cdot,d,p)\right]{\rm d}s\\
&&=-\frac{1}{b_1b_2}\Int_{\pa D}\left[{\bm u}(\cdot,z_j,q)\cdot T{\bm u}(\cdot,d,p)
-{\bm u}(\cdot,d,p)\cdot T{\bm u}(\cdot,z_j,q)\right]{\rm d}s \\
&&=-\frac{1}{b_1b_2}\Int_D\left[{\bm u}(\cdot,z_j,q)\cdot(\na\cdot( \mathcal{C}:\na{\bm u}(\cdot,d,p)))
-{\bm u}(\cdot,d,p)\cdot(\na\cdot(\mathcal{C}:\na{\bm u}(\cdot,z_j,q)))\right]{\rm d}x\\
&&=0,
\enn
that is, (\ref{4.5b}) holds.
	
Since ${\bm H}^{s}_{\infty}(\hat{x},d,p)=\wi{\bm H}^{s}_{\infty}(\hat{ x},d,p)$ for all $\hat{x},d\in\Sp^2$ 
and $p\in\R^3$, we obtain by (\ref{2.6a}) and Rellich's lemma that for each $p\in\R^3$, 
\ben
{\bm E}^{s}(x,d,p)=\wi{\bm E}^{s}(x,d,p),\quad {\bm H}^{s}(x,d,p)=\wi{\bm H}^{s}(x,d,p)\quad\;x\in G_0,
\enn
where $G_0$ denotes the unbounded component of $\mat{R}^3\se\ov{(D\cup\wi{D})}$. 
This, together with the mixed reciprocity relations (\ref{4.5b}) and (\ref{4.6b}) and Rellich's lemma again, 
implies that for each $q\in\R^3$,
\be\label{4.11b}
{\bm E}^s(x,z_j,q)=\wi{\bm E}^s(x,z_j,q),\quad{\bm H}^s(x,z_j,q)=\wi{\bm H}^s(x,z_j,q)\quad\;x\in G_0.
\en

We now prove the uniform boundedness in an appropriate Sobolev space of both ${\bm H}^s(\cdot,z_j,q)$ 
and ${\bm u}(\cdot,z_j,q)$ as $j\to\infty$. To this end, define the function
\be\label{4.12b}
\hat{\bm H}_j(x):={\bm H}^s(x,z_j,q)-\curl[q\Phi(x,y_j)]/\|\curl[q\Phi(x,z_j)]\|_{{\bm L}^2(\pa D)},
\;x\in\R^3\se\ov{D},\;\;\quad
\en
where $y_j:=z_*-(\delta/j)\nu(z_*)\in D$ for $j\in\N$. It is easy to verify that 
$({\bm u}(\cdot,z_j,q),\hat{\bm H}_j(\cdot))$ satisfies the scattering problem
\be\label{4.13bb}
\left\{
\begin{array}{ll}
\ds\curl\curl\hat{\bm H}_j -\ka^2\hat{\bm H}_j=0& \qquad{\rm in\;}\;\;D^c,\\[1mm]
\ds\na\cdot(\mathcal{C}:\na{\bm u}(\cdot,z_j,q))+\rho\om^2{\bm u}(\cdot,z_j,q)=0&\qquad{\rm in\;}\;\;D,\\[1mm]
\ds T{\bm u}(\cdot,z_j,q))-b_1\nu\times\hat{\bm H}_j={\bm f}_{1j} &\qquad{\rm on\;}\;\;\pa D,\\[1mm]
\ds\nu\times\curl\hat{\bm H}_j+\frac{{\rm i}\ka}{b_2}\nu\times{\bm u}(\cdot,z_j,q))
={\rm i}\ka{\bm f}_{2j} &\qquad{\rm on\;}\;\;\pa D,\\[1mm]
\ds\lim_{r\to\infty}r(\hat{x}\times\curl\hat{\bm H}_j+{\rm i}\ka\hat{\bm H}_j)=0, &\qquad r=|x|,
\end{array}
\right.
\en
where the data ${\bm f}_{1j}$ and ${\bm f}_{2j}$ are given by
\ben
{\bm f}_{1j}(x)&:=&\nu\ti{\bm H}^i(x,z_j,q)+\frac{\nu\times\curl[q\Phi(x,y_j)]}
{\|\curl[q\Phi(x,z_j)]\|_{{\bm L}^2(\pa D)}} \\
&=&\frac{\nu\times\curl[q\Phi(x, z_j)]}{\|\curl[q\Phi(x,z_j)]\|_{{\bm L}^2(\pa D)}}
+\frac{\nu\times\curl[q\Phi(x, y_j)]}{\|\curl[q\Phi(x,z_j)]\|_{{\bm L}^2(\pa D)}},\\[1.5mm]
{\bm f}_{2j}(x)&:=&\nu\times{\bm E}^i(x,z_j,q)-\frac{1}{{\rm i}\ka}
\frac{\nu\times\curl\curl[q\Phi( x,y_j)]}{\|\curl[q\Phi( x,z_j)]\|_{{\bm L}^2(\pa D)}}\\
&=&-\frac{1}{{\rm i}\ka}\left(\frac{\nu\times\curl\curl[q\Phi(x,z_j)]}{\|\curl[q\Phi(x,z_j)]\|_{{\bm L}^2(\pa D)}}
+\frac{\nu\times\curl\curl[q\Phi(x,y_j)]}{\|\curl[q\Phi(x,z_j)]\|_{{\bm L}^2(\pa D)}}\right).
\enn
Noting that $\Div(\nu\times{\bm f})=-\nu\cdot\curl{\bm f}$ and $\curl\curl{\bm f}=(-\Delta+\na\dive){\bm f}$ 
for a smooth function ${\bm f}$, one immediately has
\be\label{4.13b}
\|{\bm f}_{1j}\|_{{\bm L}^2(\pa D)} + \|\Div{\bm f}_{2j}\|_{{L}^2(\pa D)}\le C_1
\en
uniformly for all $j\in\N$, where $C_1>0$ is a fixed positive constant.
Moreover, by the definition of $z_j$ and $y_j$, and on taking $q:=\nu(z_*)$, we can further prove that 
${\bm f}_{2j}$ is uniformly bounded in ${\bm L}^2(\pa D)$ for all $j\in\N$, that is,
\be\label{4.14b}
\|{\bm f}_{2j}\|_{{\bm L}^2(\pa D)}\le C_2
\en
for some fixed constant $C_2>0$. In fact, a direct calculation shows that
\be\no
&&\frac{1}{\|\curl[q\Phi(x,z_j)]\|_{{\bm L}^2(\pa D)}}\{\nu\times\curl^2[q\Phi(x,z_j)]
+\nu\times\curl^2[q\Phi(x, y_j)]\}\\ \no
&&=\frac{\nu\times\grad\dive[q\Phi(x,z_j)+q\Phi(x, y_j)]+\ka^2\nu\times[q\Phi(x,z_j)
+q\Phi(x,y_j)]}{\|\curl[q\Phi(x,z_j)]\|_{{\bm L}^2(\pa D)}}\\ \no
&&=\frac{\nu\times\grad\grad[\Phi(x,z_j)+\Phi(x,y_j)]q}{\|\curl[q\Phi(x,z_j)]\|_{{\bm L}^2(\pa D)}}
+\frac{\ka^2\nu\times q[\Phi(x,z_j)+\Phi(x,y_j)]}{\|\curl[q\Phi(x,z_j)]\|_{{\bm L}^2(\pa D)}}\\ \label{4.15b}
&&=:I_j+II_j.
\en
It is easy to see that $II_j\in{\bm L}^2(\pa D)$ is uniformly bounded for $j\in\mat{N}$ since 
the fundamental solution $\Phi(\cdot,\cdot)$ is weakly singular.
To estimate $I_j$, without loss of generality, we may take $z^*=(0,0,0)^T$ and $\nu(z^*)=(0,0,1)^T$. 
Since $q=\nu(z^*)$, we have
\ben
I_j&=&\frac{1}{\|\curl(q\Phi(x,z_j))\|_{{\bm L}^2(\pa D)}}\nu(x)\ti
\begin{pmatrix}
\pa^2_{13}\Phi(x,z_j)+\pa^2_{13}\Phi(x,y_j)\\
\pa^2_{23}\Phi(x,z_j)+\pa^2_{23}\Phi(x,y_j)\\
\pa^2_{33}\Phi(x,z_j)+\pa^2_{33}\Phi(x,y_j)
\end{pmatrix}\\
&=&\frac{1}{\|\curl(q\Phi(x, z_j))\|_{{\bm L}^2(\pa D)}}\nu(z_*)\ti
\begin{pmatrix}
	\pa^2_{13}\Phi(x,z_j)+\pa^2_{13}\Phi(x,y_j)\\
	\pa^2_{23}\Phi(x,z_j)+\pa^2_{23}\Phi(x,y_j)\\
	\pa^2_{33}\Phi(x,z_j)+\pa^2_{33}\Phi(x,y_j)
\end{pmatrix}\\
&&+\frac{1}{\|\curl(q\Phi(x,z_j))\|_{{\bm L}^2(\pa D)}}(\nu(x)-\nu(z_*))\ti
\begin{pmatrix}
	\pa^2_{13}\Phi(x,z_j)+\pa^2_{13}\Phi(x,y_j)\\
	\pa^2_{23}\Phi(x,z_j)+\pa^2_{23}\Phi(x,y_j)\\
	\pa^2_{33}\Phi(x,z_j)+\pa^2_{33}\Phi(x,y_j)
\end{pmatrix}\\
&=&\frac{1}{\|\curl(q\Phi(x, z_j))\|_{{\bm L}^2(\pa D)}}
\begin{pmatrix}
	-\pa^2_{23}\Phi(x,z_j)-\pa^2_{23}\Phi(x,y_j) \\
	\pa^2_{13}\Phi(x,z_j)+\pa^2_{13}\Phi(x,y_j)\\
	0
\end{pmatrix}\\
&&+\frac{1}{\|\curl(q\Phi(x,z_j))\|_{{\bm L}^2(\pa D)}}(\nu(x)-\nu(z_*))\ti
\begin{pmatrix}
	\pa^2_{13}\Phi(x,z_j)+\pa^2_{13}\Phi(x,y_j)\\
	\pa^2_{23}\Phi(x,z_j)+\pa^2_{23}\Phi(x,y_j)\\
	\pa^2_{33}\Phi(x,z_j)+\pa^2_{33}\Phi(x,y_j)
\end{pmatrix}\\
&=:& I_j^{(1)}+I_j^{(2)}.
\enn
Let $x=(x_1,x_2,x_3)^T$ and $z=(z^{(1)},z^{(2)},z^{(3)})^T$. Then a direct calculation gives
\be\no
4\pi^2\pa^2_{\ell3}\Phi(x,z)
&=&-\ka^2\frac{(x_3-z^{(3)})(x_\ell-z^{(\ell)})e^{{\rm i}\ka|x-z|}}{|x-z|^3}-3{\rm i}\ka \frac{(x_3-z^{(3)})(x_j-z^{(\ell)})e^{{\rm i}\ka|x-z|}}{|x-z|^4 }\\ \label{4.16b}
&&+\;3\frac{(x_3-z^{(3)})(x_\ell-z^{(\ell)})e^{{\rm i}\ka|x-z|}}{|x-z|^5}\qquad{\rm for}\;\;\ell=1,2,
\en
and
\be\no
4\pi^2\pa^2_{\ell 3}\Phi(x,z)
&=&-\ka^2\frac{(x_3-z^{(3)})^2e^{{\rm i}\ka|x-z| }}{ |x-z|^3 } + {\rm i}\ka \frac{e^{{\rm i}\ka|x-z| }}{ |x-z|^2 }
-3{\rm i}\ka \frac{(x_3-z^{(3)})^2e^{{\rm i}\ka|x-z| }}{ |x-z|^4 }\\   \label{4.17b}
&&\;-\frac{e^{{\rm i}\ka|x-z|}}{|x-z|^3}+3\frac{(x_3-z^{(3)})^2e^{{\rm i}\ka|x-z|}}{|x-z|^5}\qquad{\rm for}\;\;\ell=3.
\en
Since $\pa D$ is $C^2$ smooth, we have the unit normal vector function $\nu(x)\in C^1(\pa D)$, and thus $|\nu(x)-\nu(z_*)|=O(|x-z_*|)$ for all $x\in\pa D$.
This, combined with (\ref{4.16b}), (\ref{4.17b}) and the fact that
\be\no
\|\curl(q\Phi(\cdot,z_j))\|^2_{{\bm L}^2(\pa D)}
&=&\int_{\pa D}\left|\grad\Phi(x,z_j)\ti \nu(z_*)\right|^2\;{\rm d}s\\\label{4.18b}
&\cong&\int_{\pa D}\frac{1}{|x-z_j|^4}\;{\rm d}s+O(1),
\en
implies that $I^{(2)}_j\in {\bm L}^2(\pa D)$ is uniformly bounded for all $j\in\mat{N}$. 
It remains to show that $I^{(1)}_j\in {\bm L}^2(\pa D)$ is uniformly bounded for all $j\in\mat{N}$.
From the definition of $I_j^{(1)}$ and the equality (\ref{4.16b}), it is sufficient to prove this fact 
for the first component of $I_j^{(1)}$. In view of (\ref{4.16b}) and (\ref{4.18b}), we only need to 
show that the sequence
\be\label{4.19b}
\frac{1}{\|\curl[q\Phi(x,z_j)]\|_{{\bm L}^2(\pa D)}}
\left(\frac{(x_3-z_j^{(3)})x_2e^{{\rm i}\ka|x-z_j|}}{|x-z_j|^5}
+\frac{(x_3-y_j^{(3)})x_2e^{{\rm i}\ka|x-y_j|}}{|x-y_j|^5}\right)\qquad
\en
is uniformly bounded in ${\bm L}^2(\pa D)$ for all $j\in\mat{N}$.
	
For $x=(x_1,x_2,x_3)^T$, $y=(y^{(1)},y^{(2)},y^{(3)})^T$ and $z=(z^{(1)},z^{(2)},z^{(3)})^T$, 
by a direct calculation we have
\be\no
&&\frac{(x_3-z^{(3)})x_2e^{{\rm i}\ka|x-z|}}{|x-z|^5}+\frac{(x_3-y^{(3)})x_2e^{{\rm i}\ka|x-y|}}{|x-y|^5}\\ \no
&&=\frac{(x_3-z^{(3)})x_2}{|x-z|^4}\left(\frac{e^{{\rm i}\ka|x-z|}}{|x-z|}-\frac{e^{{\rm i}\ka|x-y|}}{|x-y|}\right)
+\frac{(x_3-z^{(3)}+x_3-y^{(3)})x_2e^{{\rm i}\ka|x-y|}}{|x-y|^5}\\\no
&&+\;\left(\frac{(x_3-z^{(3)})x_2e^{{\rm i}\ka|x-y|}}{|x-z|^2|x-y|^2}
+\frac{(x_3-z^{(3)})x_2e^{{\rm i}\ka|x-y|}}{|x-z||x-y|^3}\right)
\left(\frac{1}{|x-z|}-\frac{1}{|x-y|}\right)\\ \label{4.20b}
&&+\;\left(\frac{(x_3-z^{(3)})x_2e^{{\rm i}\ka|x-y|}}{|x-y|^4}
+\frac{(x_3-z^{(3)})x_2e^{{\rm i}\ka|x-y|}}{|x-z|^3|x-y|}\right)\left(\frac{1}{|x-z|}-\frac{1}{|x-y|}\right)\qquad
\en
It follows from \cite{AN08} that there exists a constant $C_3>0$ such that
\be\label{4.21b}
\left|\frac{1}{|x-z_j|}-\frac{1}{|x-y_j|}\right|\le C_3,\quad|\Phi(x,z_j)-\Phi(x,y_j)|\le C_3,\qquad x\in\pa D.
\en
Recalling $z_*=(0,0,0)^T$ and $\nu(z_*)=(0,0,1)^T$, we deduce by Taylor's expansion that 
there exists a constant $C_4>0$ such that
\be\label{4.22b}
|x_3|\le C_4(x_1^2+x_2^2)\quad{\rm for\;all}\;x\in\pa D.
\en
Inserting (\ref{4.21b}) and (\ref{4.22b}) into (\ref{4.20b}) and using (\ref{4.18b}), we obtain that 
the sequence in (\ref{4.19b}) is uniformly bounded in ${\bm L}^2(\pa D)$ for all $j\in\N$, 
which means that $I_j^{(1)}$ is uniformly bounded in ${\bm L}^2(\pa D)$ for all $j\in\N$. 
Thus, and by (\ref{4.15b}), the inequality (\ref{4.14b}) holds.
By (\ref{4.13b}), (\ref{4.14b}) and Theorem \ref{thm3.9} for the problem (\ref{4.13bb}), 
we obtain the uniform estimate
\be\no
&&\|\hat{\bm H}_j(\cdot)\|_{H(\curl,B_R\se\ov{D})}+\|{\bm u}(\cdot,z_j,q)\|_{{\bm H}^1(D)}\\ \label{4.23b}
&&\qquad\qquad\qquad\qquad\quad\leq C_4(\|{\bm f}_{1j}\|_{{\bm H}^{-1/2}(\pa D)}
+\|{\bm f}_{2j}\|_{H^{-1/2}_{\Div}(\pa D)})\le C_5,
\en
where $C_4,C_5>0$ are two positive constants independent of $j\in\N$.
	
Since $\pa D\in C^2$ and $D_1\not= D_2$, we can choose a small subdomain of $D$, denoted by $D_0$ 
with a $C^2$-smooth boundary $\pa D_0$, satisfying that $D_0\subset D\se\ov{\wi{D}}$
and the intersection $B\cap\pa D_0$ contains an open segment of $\pa D$. 
In $D_0$, we can construct the following boundary value problem
\be\label{4.24b}
\left\{
\begin{array}{ll}
\ds\curl\curl {\bm F}_{0,j} +{\bm F}_{0,j} =\xi_{1,j}& \quad{\rm in\;}\;D_0,\\[1mm]
\ds\na\cdot( \mathcal{C}:\na{\bm E}_{0,j}) -{\bm E}_{0,j}=\xi_{2,j} &\quad {\rm in\;}\; D_0,\\[1mm]
\ds T{\bm E}_{0,j}-b_1\nu \ti{\bm F}_{0,j} ={\bm h}_{1,j} &\quad {\rm on\;}\;\pa D_0,\\[1mm]
\ds \nu\times\curl{\bm F}_{0,j}+\frac{{\rm i}\ka}{b_2}\nu\times{\bm E}_{0,j}
={\bm h}_{2,j}& \quad{\rm on\;}\;\pa D_0,
\end{array}
\right.
\en
with $\xi_{1,j},\xi_{2,j},{\bm h}_{1,j}$ and $ {\bm h}_{2,j}$ defined as follows:
\ben
\xi_{1,j}: &=& (\ka^2+1)\wi{\bm H}(\cdot,z_j,q),\\
\xi_{2,j}: &=& (\rho\om^2+1){\bm u}(\cdot,z_j,q),\\
{\bm h}_{1,j}: &=&T{\bm u}(\cdot,z_j,q)-  b_1\nu\ti\wi{\bm H}(\cdot,z_j,q),\\
{\bm h}_{2,j}: &=& \nu\ti\curl\wi{\bm H}(\cdot,z_j,q)+\frac{{\rm i}\ka}{b_2} \nu\ti{\bm u}(\cdot,z_j,q).
\enn
By Theorem \ref{thm3.9} it can be verified that all data are well-defined in the related Sobolev spaces 
for each fixed $j\in\N$. By Lemma \ref{lem4.1}, the problem (\ref{4.24b}) is uniquely solvable 
with the estimate
\be\no
\|{\bm F}_{0,j}\|_{H(\curl,D_0)}+\|{\bm E}_{0,j}\|_{{\bm H}^1(D_0)}
&\leq &C_6\left[\|\xi_{1,j}\|_{{\bm L}^2(D_0)}+\|\xi_{2,j}\|_{{\bm L}^2(D_0)}\right.\\ \label{4.25b}
&&\quad\left. +\|{\bm h}_{1,j}\|_{{\bm H}^{-\frac{1}{2}}(\pa D_0)} 
+\|{\bm h}_{2,j}\|_{H^{-\frac{1}{2}}_{\Div}(\pa D_0)}\right],\qquad
\en
where $C_6>0$ is a fixed constant.
	
We now claim that the right-hand side of the inequality (\ref{4.25b}) is uniformly bounded for all $j\in\N$. 
In fact, it first follows from (\ref{4.23b}) that $\xi_{2,j}$ is uniformly bounded in
${\bm L}^2(D_0)$ for all $j\in\N$. Due to the positive distance between $D_0$ and $\wi{D}$, 
it can be concluded by the well-posedness of the problem (\ref{3.2a}) associated with the elastic body
$\wi{D}$ and the uniform boundedness of ${\bm H}^i(\cdot,z_j,q)$ in ${\bm L}^2(D_0)$ 
that $\xi_{1,j}$ is uniformly bounded in ${\bm L}^2(D_0)$ for all $j\in\N$.
By recalling the equality (\ref{4.11b}) and the transmission conditions (\ref{2.5a}) and (\ref{2.5b}), 
it is deduced that
\ben
{\bm h}_{1,j}=0,\qquad {\bm h}_{2,j}=0\quad {\rm on\;}\;\;\pa D_0\cap\pa D
\enn
for all $j\in\N$. Thus we only need to show that ${\bm h}_{1,j}$ and ${\bm h}_{2,j}$ are uniformly 
bounded in ${\bm H}^{-1/2}(\G)$ and ${\bm H}^{-1/2}_{\Div}(\G)$, respectively, where
$\G:=\pa D_0\se\ov{B_\vep(z_*)}$ and $B_\vep(z_*)$ is a ball with sufficiently small radius $\vep>0$ 
such that $B_\vep(z_*)\subset D_0$.
Since $\wi{\bm H}(\cdot,z_j,q)={\bm H}^i(\cdot,z_j,q)+\wi{\bm H}^s(\cdot,z_j,q)$, and by (\ref{4.23b}), 
we deduce that
\ben
\|{\bm u}(\cdot,z_j,q)\|_{{\bm H}^1(D_0\se\ov{B_\vep(z_*)})}
+\|\wi{\bm H}(\cdot,z_j,q)\|_{H(\curl,D_0\se\ov{B_\vep(z_*)})}\le C_7
\enn
for a some constant $C_7>0$, whence the uniform boundedness of 
$\|{\bm h}_{1,j}\|_{{\bm H}^{-1/2}(\pa D_0)}$ and $\|{\bm h}_{2,j}\|_{H^{-1/2}_{\Div}(\pa D_0)}$ 
follows from the trace theorems.
	
It is easy to verify that $({\bm F}_{0,j},{\bm E}_{0,j}):=(\wi{\bm H}(\cdot,z_j,q),{\bm u}(\cdot,z_j,q))$ 
is the unique solution to the problem (\ref{4.24b}). Then, by (\ref{4.25b}) we have
\be\label{4.26}
\|\wi{\bm H}(\cdot,z_j,q)\|_{H(\curl,D_0)}\le C_8\;\;\;\forall\;j\in\N
\en
for some constant $C_8>0$ independent of $j\in\N$. On the other hand, due to 
the positive distance between $z_*$ and $\wi{D}$, we have 
\be\label{4.27}
\|\wi{\bm H}^s(\cdot,z_j,q)\|_{H(\curl,D_0)}\le C_9\;\;\;\forall\;j\in\N
\en
for some constant $C_9>0$ independent of $j\in\N$. From (\ref{4.27}) it follows that 
\ben
\|\wi{\bm H}(\cdot,z_j,q)\|_{H(\curl,D_0)}
&=&\|{\bm H}^i(\cdot,z_j,q)+\wi{\bm H}^s(\cdot,z_j,q)\|_{H(\curl,D_0)}\\
&\ge&\|{\bm H}^i(\cdot,z_j,q)\|_{H(\curl,D_0)}-\|\wi{\bm H}^s(\cdot,z_j,q)\|_{H(\curl,D_0)}\\
&\ge&\|{\bm H}^i(\cdot,z_j,q)\|_{H(\curl,D_0)}-C_9.
\enn
By \cite[Theorem 3.8]{JBH18} it is seen that the right-hand side of the above inequality goes to infinity 
as $j\to\infty$, which contradicts to the inequality (\ref{4.26}), meaning that $D=\wi{D}$.
The proof is thus complete.
\end{proof}

\section*{Acknowledgements}

This work is partially supported by the NNSF of China Grants No. 11771349, 91730306, 91630309 
and Fundamental Research Funds for the Central Universities of China.

\end{document}